\documentclass[oneside,english]{amsart}
\usepackage[T1]{fontenc}
\usepackage[latin9]{inputenc}
\usepackage{geometry}
\usepackage{color}
\usepackage{babel}
\usepackage{refstyle}
\usepackage{units}
\usepackage{amsthm}
\usepackage{amssymb}
\usepackage[all]{xy}
\usepackage[unicode=true,pdfusetitle,
 bookmarks=true,bookmarksnumbered=false,bookmarksopen=false,
 breaklinks=false,pdfborder={0 0 0},backref=false,colorlinks=true]
 {hyperref}
\hypersetup{
 citecolor=blue}
\usepackage{breakurl}

\makeatletter

\AtBeginDocument{\providecommand\secref[1]{\ref{sec:#1}}}
\AtBeginDocument{\providecommand\lemref[1]{\ref{lem:#1}}}
\AtBeginDocument{\providecommand\thmref[1]{\ref{thm:#1}}}
\AtBeginDocument{\providecommand\subref[1]{\ref{sub:#1}}}
\AtBeginDocument{\providecommand\remref[1]{\ref{rem:#1}}}
\AtBeginDocument{\providecommand\propref[1]{\ref{prop:#1}}}
\AtBeginDocument{\providecommand\corref[1]{\ref{cor:#1}}}
\RS@ifundefined{subref}
  {\def\RSsubtxt{section~}\newref{sub}{name = \RSsubtxt}}
  {}
\RS@ifundefined{thmref}
  {\def\RSthmtxt{theorem~}\newref{thm}{name = \RSthmtxt}}
  {}
\RS@ifundefined{lemref}
  {\def\RSlemtxt{lemma~}\newref{lem}{name = \RSlemtxt}}
  {}

\numberwithin{equation}{section}
\numberwithin{figure}{section}
\theoremstyle{plain}
\newtheorem{thm}{\protect\theoremname}[section]
\theoremstyle{plain}
\newtheorem{cor}[thm]{\protect\corollaryname}
\theoremstyle{definition}
\newtheorem{defn}[thm]{\protect\definitionname}
\theoremstyle{remark}
\newtheorem{rem}[thm]{\protect\remarkname}
\ifx\proof\undefined
\newenvironment{proof}[1][\protect\proofname]{\par
\normalfont\topsep6\p@\@plus6\p@\relax
\trivlist
\itemindent\parindent
\item[\hskip\labelsep\scshape #1]\ignorespaces
}{%
\endtrivlist\@endpefalse
}
\providecommand{\proofname}{Proof}
\fi
\theoremstyle{plain}
\newtheorem{prop}[thm]{\protect\propositionname}
\theoremstyle{plain}
\newtheorem{lem}[thm]{\protect\lemmaname}

\@ifundefined{date}{}{\date{}}
\newref{def}{name=Definition~}
\newref{ex}{name=Example~}
\newref{sub}{name=Subsection~}
\newref{cor}{name=Corollary~}
\newref{rem}{name=Remark~}
\newref{prop}{name=Proposition~}
\newref{lem}{name=Lemma~}
\newref{thm}{name=Theorem~}
\newref{sec}{name=Section~}

\hypersetup{citecolor=blue}

\makeatother

\providecommand{\corollaryname}{Corollary}
\providecommand{\definitionname}{Definition}
\providecommand{\lemmaname}{Lemma}
\providecommand{\propositionname}{Proposition}
\providecommand{\remarkname}{Remark}
\providecommand{\theoremname}{Theorem}

\begin{document}
\global\long\def\CC{\mathbb{C}}
\global\long\def\FF{\mathbb{F}}
\global\long\def\EE{\mathbb{E}}
\global\long\def\KK{\mathbb{K}}
\global\long\def\LL{\mathbb{L}}
\global\long\def\NN{\mathbb{N}}
\global\long\def\QQ{\mathbb{Q}}
\global\long\def\RR{\mathbb{R}}
\global\long\def\ZZ{\mathbb{Z}}

\title{The center of the generic $G$-crossed product}

\author{Ofir David}


\address{Department of Mathematics, Technion-Israel Institute of Technology,
Haifa 32000, Israel}

\email{eofirdav@tx.technion.ac.il}

\thanks{The author was supported by THE ISRAEL SCIENCE FOUNDATION (grant
No. 1017/12).}
\begin{abstract}
Let $G$ be a finite group and let $\FF$ be a field of characteristic
zero. In this paper we construct a generic $G$-crossed product over
$\FF$ using generic graded matrices. The center of this generic $G$-crossed
product, denoted by $\FF(G)$, is then the invariant field of a suitable
$G$ action on a field of rational functions in several indeterminates.
The main goal of this paper is to study the extensions $\nicefrac{\FF(G)}{\FF}$
given that $\FF$ contains enough roots of unity and determine how
close they are to being purely transcendental.

In particular we show that $\FF(G)/\FF$ is a stably rational extension
for $G=C_{2}\times C_{2n}$ where $n$ is odd and for $G=\left\langle \sigma,\tau\;\mid\;\sigma^{n}=\tau^{2m}=e,\;\tau\sigma\tau^{-1}=\sigma^{-1}\right\rangle $
where $gcd(n,2m)=1$. Furthermore, we prove that if $H,K$ are groups
of coprime orders, then $\FF(H\times K)$ is the fraction field of
$\FF(H)\times\FF(K)$.
\end{abstract}

\maketitle

\section{Introduction}

The Brauer group $Br(\FF)$ of a field $\FF$ is considered as one
of its most important arithmetic invariants. It consists of all the
finite dimensional division algebras central over $\FF$ (or equivalently,
central simple algebras up to Morita equivalence). An important tool
used in studying this group, and more generally the groups $Br(\LL)$
for field extensions $\LL/\FF$, is the generic division algebra $D_{n}$
of degree $n$ over the field $\FF$. It has the remarkable property
that any other central simple algebra of degree $n$ of a field extension
of $\FF$ is a specialization of $D_{n}$. The generic algebra's usefulness
comes mainly from the fact that many nice properties satisfied by
it are inherited by all of its specializations (see \cite{saltman_lectures_1999}
for details). For example, if $D_{n}$ is a crossed product with a
group $G$, then any division algebra of degree $n$ over a field
extension of $\FF$ is also a crossed product with the same group
$G$. Amitsur used this property to show that for suitable integers
$n$ the generic division algebra $D_{n}$ is not a crossed product.
This was the first example found of a noncrossed product division
algebra (see \cite{amitsur_central_1972}).

Let us recall another such property. A central simple algebra is called
cyclic if it is isomorphic to a crossed product with a cyclic group.
The Merkurjev-Suslin theorem \cite{merkurjev_cohomology_1982} states
that over a field $\FF$ containing enough roots of unity, every central
simple algebra is equivalent to a tensor product of cyclic algebras,
or in other words $Br(\FF)$ is generated by the equivalence classes
of cyclic algebras. While Merkurjev-Suslin's proof uses K-cohomology,
another approach to this problem, given prior to their proof, uses
generic algebras. It is not difficult to show that if $D_{n}$ is
equivalent to a product of cyclic algebras, then so is every central
simple algebra of degree $n$ over a field extension of $\FF$. Consider
an algebraic field extension $\FF$ of $\QQ$ containing all roots
of unity. Then it is known that $Br(\FF)$ is trivial, namely the
only central simple algebras over it are matrix algebras, and in particular
$Br(\FF)$ is generated by classes of cyclic algebras (in a trivial
way). Furthermore, by a well known theorem due to Bloch, the property
of generation by classes of cyclic algebras extends to $Br(\FF(x_{1},...,x_{n}))$
whenever the $x_{i}$ are algebraically independent over $\FF$. Using
these ideas, if one could show that the center of $D_{n}$ is a rational
extension of $\FF$, then it will follow that any central simple algebra
of degree $n$ over a field extension of $\FF$ is equivalent to a
product of cyclic algebras, thus proving Merkurjev-Suslin's theorem
for such fields (for more details and a proof of Bloch's theorem
using the Auslander Brumer Faddeev theorem see \cite{fein_brauer_1981},
and for Bloch's original proof see \cite{bloch_torsion_1974}).

The last argument demonstrates the significance of the center $Z_{n}:=Z(D_{n})$
of $D_{n}$ in the study of Brauer groups. In particular, we are interested
in determining if the field extension $Z_{n}/\FF$ is rational (i.e.
purely transcendental), or has some weaker rationality property, namely
it is stably rational, retract rational or just unirational (the definitions
for the different types of rationality extensions are given in \secref{Preliminaries}).

Over the rational field $\QQ$, the first result in this direction
is attributed to Sylvester, who showed in 1883 that $Z_{2}/\QQ$ is
rational (see \cite{sylvester_involution_1883}). About 80 years later,
using the algebra of generic matrices to construct the generic algebra,
Procesi was able to show that $Z_{n}$ is the invariant field of a
suitable $S_{n}$ action on a rational extension of $\QQ$, hence
in particular these field extensions are unirational (see \cite{procesi_non-commutative_1967}).
Applying Procesi's method, Formanek proved the rationality for $n=3,4$
(see \cite{formanek_center_1979,formanek_center_1980} and also \cite{formanek_polynomial_1992}).
Le Bruyn and Bessenrodt in \cite{bessenrodt_stable_1991} proved that
for $n=5,7$ the field $Z_{n}$ is a stably rational extension of
$\QQ$ and Beneish in \cite{beneish_induction_1998} gave a more elementary
proof for these primes. Schofield \cite{schofield_matrix_1992}, Katsylo
\cite{katsylo_stable_1990} and Saltman \cite{saltman_note_1992}
showed that $Z_{nm}$ is stably rationally equivalent to the fraction
field of $Z_{n}\otimes Z_{m}$ whenever $n,m$ are coprime, thus reducing
the problem to the prime power case. Finally, Saltman showed in \cite{saltman_retract_1984}
that $Z_{p}$ is retract rational over $\FF$ for $p$ prime, and
this result can be extended to product of distinct primes. For more
information see Le Bruyn's survey \cite{le_bruyn_centers_1991}.

It is well known that any central simple algebra is Brauer equivalent
to a $G$-crossed product for some finite group $G$, hence it is
only natural to consider generic $G$-crossed products. Indeed, several
equivalent constructions were given by Snider \cite{snider_is_1979}
and Rosset \cite{rosset_group_1978}, who used relation modules, and
by Saltman \cite{saltman_lectures_1999,saltman_retract_1984}, who
used generic $2$-cocycles. As with the generic algebra case, many
properties of the generic crossed product are inherited to all the
$G$-crossed products, and in particular the idea mentioned above
for proving the Merkurjev-Suslin theorem is still applicable here
(see for example \cite{snider_is_1979}).

With this goal in mind, Snider proved that the center of the generic
$G$-crossed product is a rational extension over $\QQ$ where $G$
is the Klein four group and stably rational for Dihedral groups of
order $2n$ over $\QQ\left(\zeta_{2n}\right)$ where $\zeta_{2n}$
is a primitive $2n$-root of unity. Another result due to Saltman
shows that if all the $p$-Sylow subgroups of $G$ are cyclic then
the center of the generic $G$-crossed product is retract rational
over a suitable extension of the base field. Recall that a group $G$
such that all of its $p$-Sylow subgroups are cyclic is a semidirect
product of cyclic groups $C_{n}\rtimes C_{m}$ where $gcd(n,m)=1$
(see \cite{hall_theory_1999}, section 9.4). This family contains
the affine groups $\FF_{q}\rtimes\FF_{q}^{\times}$ of finite fields
which appear in Beneish's work, where she proves that the centers
of the generic division algebras of rank 5 and 7 are stably rational
extensions over $\QQ$. On the other hand, Saltman proved that roughly
speaking these are almost all the groups for which the center may
be close to being a rational extension. More precisely, Saltman showed
that the center can be retract rational only if each $p$-Sylow subgroup
is either cyclic or a product of two cyclic groups.

The goal of this paper is to further study the centers of the generic
crossed products. For a group $G$ and a field $\FF$, denote by $\FF(G)$
the center of the generic $G$-crossed product over the field $\FF$.
In what follows, we always assume that $\FF$ contains a primitive
root of unity of order $\left|G\right|$. The first result deals with
groups containing only cyclic $p$-Sylow subgroups.
\begin{thm}
\label{thm:Main_Theorem_sylow}Let $G=C_{n}\rtimes C_{m},\; gcd(m,n)=1$,
be a group with cyclic $p$-Sylow subgroup. Then:
\begin{enumerate}
\item The extension $\FF(G)/\FF$ is retract rational.
\item If $K=Z(G)\cap C_{m}$ is the kernel of the action of $C_{m}$ on
$C_{n}$ by conjugation, then the extension $\FF(G)/\FF$ is stably
isomorphic to $\FF(\nicefrac{G}{K})/\FF$.
\end{enumerate}
\end{thm}
Combining Snider's result for Dihedral groups and part (2) of the
theorem above we conclude the following.
\begin{cor}
\label{cor:generalized_diehedral}Let. $G=\left\langle \sigma,\tau\;\mid\;\sigma^{n}=\tau^{2m}=e,\;\tau\sigma\tau^{-1}=\sigma^{-1}\right\rangle $
where $gcd(n,2m)=1$. Then $\FF(G)/\FF$ is a stably rational extension.
\end{cor}
Recall that for the standard generic division algebra, the question
of rationality can be reduced to the prime power case. The next result
is its counterpart for generic crossed products.
\begin{thm}
\label{thm:coprime_invariants}Let $G\cong H\times K$ where $H,K$
are groups of coprime orders. Then $\FF(G)/\FF$ is stably rationally
equivalent to the fraction field of $\FF(H)\otimes\FF(K)$.
\end{thm}
Using Snider's result for the Klein four group we conclude the following.
\begin{cor}
The extension $\FF(C_{2}\times C_{2n})/\FF$ is stably rational whenever
$n$ is odd and $\FF$ contains a primitive $n$-th root of unity.
\end{cor}
Interestingly, it is still unknown whether $\FF(G)/\FF$ is stably
rational for $G=C_{p}\times C_{p}$ with $p\geq3$ prime.\\

The paper is organized as follows.

We start with the construction of the generic crossed product in \secref{Generic-Crossed-Products}
using generic graded matrices and give a description of its center
using $G$-lattices. In \secref{Preliminaries} we recall the definitions
and main results needed from the theory of $G$-lattices and their
field invariants.

In \secref{Flows-in-Graphs} we introduce the concept of flows in
graphs and use them in order to construct the $G$-lattices appearing
both in study of center of the generic division algebra and generic
crossed product. The main results of this paper are proved in \secref{The-center}
where we restrict to the flows representing the center of generic
crossed products. Finally, in \secref{The-ungraded-center} we recall
some results on the center of the standard generic division algebra
and interpret them using the language of flows in graphs.

\section{\label{sec:Generic-Crossed-Products}Generic Crossed Products}

Fix a field $\FF$ of characteristic zero. All the rings in this section
are algebras over $\FF$, and the homomorphism are always $\FF$-homomorphisms.

The generic division algebra and its center have been extensively
studied in the literature. A detailed account can be found in \cite{formanek_polynomial_1992}
and in Le Bruyn's survey \cite{le_bruyn_centers_1991} which details
the motivation and different approaches to study the center of these
generic algebras. An analogous object can be defined for the class
of $G$-crossed products for a fixed finite group $G$. It has several
equivalent constructions in the literature using relation modules
by Snider and Rosset \cite{snider_is_1979,rosset_group_1978} and
generic 2-cocycle by Saltman \cite{saltman_retract_1984,saltman_lectures_1999}.
A more general approach using graded polynomial identities was studied
by Aljadeff and Karasik \cite{aljadeff_polynomial_????}. The polynomial
identities construction was first used for the standard generic division
algebra, and can be further generalized to generic algebras for other
classes of graded simple algebra (which include crossed product) and
even to generic Hopf algebras \cite{aljadeff_polynomial_2008}. The
approach utilized here uses generic graded matrices which is very
similar to the one with graded polynomial identities.\\

Usually, $G$-crossed products are defined over Galois extensions
of fields. For the generic crossed product definition we require the
more generalized definition of Galois extensions of rings, which we
briefly describe. For a full treatment of Galois extension of commutative
rings and the crossed products defined over them, we refer the reader
to \cite{meyer_separable_1971} and \cite{orzech_brauer_1975}.
\begin{defn}[Galois Extension]
\label{def:Galois}Let $R\subseteq S$ be an extension of unital
commutative rings, and let $G$ be a finite subgroup of $Aut(S)$.
We say that $S$ is a \emph{$G$-Galois extension} of $R$ if
\begin{enumerate}
\item $S$ is a faithful $R$ algebra.
\item $S^{G}=R$.
\item There are $x_{1},...,x_{n},y_{1},...,y_{n}\in S$ for some $n\in\NN$
such that $\sum_{i}x_{i}g(y_{i})=\delta_{e,g}$ where $\delta_{e,g}=1$
if $g=e$ and zero otherwise. 
\end{enumerate}
\end{defn}

\begin{defn}[$G$-Crossed Product]
\label{def:crossed_product}Let $S$ be a $G$-Galois extension of
$R$ and let $c\in Z^{2}(G,S^{\times})$ be a 2-cocycle, namely $c$
is a function from $G\times G$ to $S^{\times}$ that satisfies
\begin{eqnarray*}
\mbox{for all }g_{1},g_{2},g_{3}\in G & \qquad & \quad c(g_{1},g_{2})c(g_{1}g_{2},g_{3})=g_{1}\left(c(g_{2},g_{3})\right)c(g_{1},g_{2}g_{3}).
\end{eqnarray*}
Consider the $R$-module $\bigoplus_{g\in G}S\cdot\epsilon_{g}$ with
a multiplication defined by 
\[
\mbox{for all }\alpha,\beta\in S,\; g,h\in G\qquad\alpha\epsilon_{g}\beta\epsilon_{h}=\alpha g(\beta)\epsilon_{g}\epsilon_{h}=\alpha g(\beta)c(g,h)\epsilon_{gh}.
\]
This algebra is called a \emph{$G$-crossed product }and is denoted
by $\Delta(S/R,G,c)$. 
\end{defn}
Crossed products over Galois extensions of fields are essential in
the study of central simple algebras. Their generalized versions play
a similar role in the study of Azumaya algebras and have many similar
properties. In particular, cohomologous 2-cocycles produce isomorphic
crossed product, so we can assume that the two cocycles are always
normalized, namely $\epsilon_{e}$ is the unity of $\Delta(S/R,G,c)$,
and hence we identify $S$ with $S\cdot\epsilon_{e}$ and $R$ with
$R\cdot\epsilon_{e}$. Under this notation, the crossed product $\Delta(S/R,G,c)$
is an Azumaya algebra central over $S^{G}=R$.\\

Recall that a \emph{$G$-grading} of an $S$-algebra $A$, is a decomposition
$A={\displaystyle \bigoplus_{g\in G}}A_{g}$ as an $S$-module such
that $A_{g}A_{h}\subseteq A_{gh}$. Clearly, setting $\Delta_{g}=S\cdot\epsilon_{g}$
produces a $G$-grading of $\Delta:=\Delta(S/R,G,c)$. It is well
known that any crossed product over a Galois extension of fields is
central simple, and therefore after suitable scalars extension it
becomes a matrix algebra. The crossed product's natural group grading
induces a grading on that matrix algebra which we now describe.
\begin{defn}[Elementary Grading]
Let $R$ be an $\FF$-algebra, $A=M_{n}(R)$ and $\bar{g}=(g_{1},...,g_{n})\in G^{n}$
be a tuple of length $n$. The \emph{elementary grading} on $A$ induced
by $\bar{g}$ is defined by $A_{h}=span_{R}\left\{ E_{i,j}\;\mid\; g_{i}^{-1}g_{j}=h\right\} $,
where $E_{i,j}$ is the matrix with $1$ in the $(i,j)$ coordinate
and zero elsewhere. In case each element of $G$ appears exactly once
in $\bar{g}$, the induced elementary grading is called the \emph{crossed
product grading}.\end{defn}
\begin{rem}
Note that reordering the tuple $\bar{g}$ produces graded isomorphic
algebras. 
\end{rem}
Let $M_{n}(R)$ have the crossed product grading with a tuple $\bar{g}=(g_{1},...,g_{n})\in G^{n}$,
where $n=\left|G\right|$. Identifying $G$ with $\left\{ 1,...,n\right\} $,
we write $E_{g_{i},g_{j}}$ instead of $E_{i,j}$. Let $S$ be the
subalgebra of diagonal matrices in $M_{n}(R)$ and set $P_{g}$ to
be the permutation matrix $P_{g}:={\displaystyle \sum_{h\in G}}E_{h,hg}$.
These permutation matrices play the roles of $\epsilon_{g}$ in the
definition of $G$-crossed products, and since $P_{g}P_{h}=P_{gh}$,
the corresponding two cocycle is trivial. Letting $G$ act on $S$
by $g(s)=P_{g}sP_{g}^{-1}$ for $g\in G$ and $s\in S$ we get that
$S^{G}\cong R$ are the scalar matrices and $S/S^{G}$ is a $G$-Galois
extension and therefore $M_{n}(R)\cong\Delta\left(S/R,G,\mathsf{1}\right)$
with $(M_{n}(R))_{g}=S\cdot P_{g}$. Unless otherwise stated, we will
always assume that the grading on $M_{n}(R)$ is the crossed product
grading. 
\begin{thm}
If $\Delta=\Delta(S/R,G,c)$ is a $G$-crossed product where $R$
is an integral domain, then there is some field $R\subseteq\LL$ such
that $\Delta\otimes_{R}\LL\cong M_{\left|G\right|}(\LL)$ as graded
algebras where $M_{\left|G\right|}(\LL)$ has the crossed product
grading.\end{thm}
\begin{proof}
This result is well known, see for example Theorem 10 in \cite{aljadeff_crossed_2013}.
For clarity we give here a sketch of the proof.

First, by extending the scalars by the field of fractions of $R$,
we may assume that $R$ is a field. Since $S/R$ is Galois with $R$
a field, we get that $S\cong\LL^{k}$ where $\LL/R$ is a field extension
and $\left[\LL:R\right]\cdot k=\left|G\right|$. Thus, extending the
scalars by $\LL$, we get the Galois extension $\LL^{n}/\LL$ where
the group $G$ acts on $\LL^{n}$ by permuting the idempotents via
the regular representation of $G$. 

Since $\LL$ is a field and $\Delta_{\LL}:=\Delta\otimes_{R}\LL$
is central simple, we have an isomorphism $\Delta_{\LL}\cong M_{\left|G\right|}(\LL)$.
Secondly, since $S\cong\LL^{k}$, we may assume (after another autmorphism)
that $S$ is the set of diagonal matrices and $R$ is the set of scalar
matrices. Under this identification, the elements $\epsilon_{g}P_{g}^{-1}$
commute with the diagonal matrices, and therefore must be diagonal,
so that $\epsilon_{g}=t_{g}P_{g}$ for some $t_{g}$ invertible and
diagonal. It follows that $\Delta_{g}=S\epsilon_{g}=S\cdot P_{g}$
is exactly the crossed product grading.
\end{proof}
With the crossed product grading in mind, we turn to construct the
generic $G$-crossed product. Define a $g$-generic matrix to be
\[
X_{g,i}=\left(\begin{array}{cccc}
x_{g_{1},g_{1}g}^{(i)} & 0 & \cdots & 0\\
0 & x_{g_{2},g_{2}g}^{(i)} &  & \vdots\\
\vdots &  & \ddots & 0\\
0 & \cdots & 0 & x_{g_{n},g_{n}g}^{(i)}
\end{array}\right)P_{g},
\]
which is homogeneous of degree $g$ in $M_{\left|G\right|}(\FF(x_{h,g}^{(i)}))$
with the crossed product grading, where the $x_{h,g}^{(i)}$ are algebraically
independent over $\FF$. Let $R_{G,k}$ be the unital $\FF$-subalgebra
of $M_{\left|G\right|}(\FF(x_{h,g}^{(i)}))$ generated by $\left\{ X_{g,i}\;\mid\; g\in G,\;1\leq i\leq k\right\} $.
The generic crossed product $D_{G,k}$ is $R_{G,k}$ after inverting
all the nonzero central elements. 

Next, we give a presentation of the $e$ component of $D_{G,k}$ as
a Galois field extension of the center of $D_{G,k}$.
\begin{thm}
Let $\LL\leq\FF\left(x_{g,h}^{(i)}\right)$ be the field extension
of $\FF$ generated by elements of the form
\[
x_{h,hg_{1}}^{(j_{1})}x_{hg_{1},hg_{1}g_{2}}^{(j_{2})}x_{hg_{1}g_{2},hg_{1}g_{2}g_{3}}^{(j_{3})}\cdots x_{h\prod_{1}^{m-1}g_{i},h\prod_{1}^{m}g_{i}}^{(j_{m})},
\]
where $m,j_{1},...,j_{m}\in\NN$, $h,g_{1},...,g_{m}\in G$ and $g_{1}\cdots g_{m}=e$.
Then $\left(D_{G,k}\right)_{e}\cong\LL$ and $Z(D_{G,k})\cong\LL^{G}$
where the $G$-action is defined by $\sigma(x_{h,g}^{(i)})=x_{\sigma h,\sigma g}^{(i)}$.\end{thm}
\begin{proof}
In order to study the center $Z(R_{G,k})$, note first that these
elements are scalar matrices, and in particular homogeneous of degree
$e$ (or equivalently diagonal). If $A=\prod X_{g_{i},j_{i}}$ is
a monomial with $\prod g_{i}=e$, then 
\begin{eqnarray*}
A_{(h,h)} & = & x_{h,hg_{1}}^{(j_{1})}x_{hg_{1},hg_{1}g_{2}}^{(j_{2})}x_{hg_{1}g_{2},hg_{1}g_{2}g_{3}}^{(j_{3})}\cdots x_{h\prod_{1}^{m-1}g_{i},h\prod_{1}^{m}g_{i}}^{(j_{m})}\\
 & = & h\left(x_{e,g_{1}}^{(j_{1})}x_{g_{1},g_{1}g_{2}}^{(j_{2})}x_{g_{1}g_{2},g_{1}g_{2}g_{3}}^{(j_{3})}\cdots x_{\prod_{1}^{m-1}g_{i},\prod_{1}^{m}g_{i}}^{(j_{m})}\right)=h(A_{(e,e)})
\end{eqnarray*}
for any $h\in H$. Since $(R_{G,k})_{e}$ is spanned by such monomials,
it follows that $A_{(h,h)}=h(A_{(e,e)})$ for any $A\in\left(R_{G,k}\right)_{e}$.
Consequently, if $A\in\left(R_{G,k}\right)_{e}$ is nonzero, then
all its entries on the diagonal are nonzero, and in particular $\left(R_{G,k}\right)_{e}$
and $Z(R_{G,k})$ are integral domains.

Define $\varphi:\left(R_{G,k}\right)_{e}\to\FF(x_{h,g}^{(i)})$ by
sending a diagonal matrix to its $(e,e)$ entry. By the preceding
argument this map is injective and contains all the elements of the
form $x_{e,g_{1}}^{(j_{1})}x_{g_{1},g_{1}g_{2}}^{(j_{2})}x_{g_{1}g_{2},g_{1}g_{2}g_{3}}^{(j_{3})}\cdots x_{\prod_{1}^{m-1}g_{i},\prod_{1}^{m}g_{i}}^{(j_{m})}$
and note that $\LL$ is exactly the fraction field of $\varphi\left(\left(R_{G,k}\right)_{e}\right)$.
Additionally an element is central in $R_{G,k}$ if and only if its
image under $\varphi$ is $G$-invariant. 

Extend $\varphi$ to $\left(D_{G,k}\right)_{e}$ in the natural way.
By definition $\varphi\left((R_{G,k})_{e}\right)\subseteq\varphi\left((D_{G,k})_{e}\right)\subseteq\LL$,
hence the theorem will be proven if we can show that $\left(D_{G,k}\right)_{e}$
is a field. On the other hand, $Z(D_{G,k})$ is a field, so it is
enough to show that $\left(D_{G,k}\right)_{e}$ is finite dimensional
over it (since it is an integral domain).

The same argument as in the beginning of the proof shows that all
the nonzero homogeneous elements in $D_{G,k}$ are invertible in $M_{\left|G\right|}(\FF(x_{h,g}^{(i)}))$.
Recall that a graded algebra $B$ is called \emph{graded prime} if
for any homogeneous elements $b_{1},b_{2}\in B$ we have $b_{1}Bb_{2}=0$
if and only if $b_{1}=0$ or $b_{2}=0$. In particular $D_{G,k}$
is graded prime and by a graded analog of Posner's theorem given in
\cite{balaba_graded_2005}, it is graded simple over its center. In
particular $D_{G,k}\otimes_{Z(D_{G,k})}\FF(x_{h,g}^{(i)})$ is graded
simple and therefore the homomorphism $D_{G,k}\otimes_{Z(D_{G,k})}\FF(x_{h,g}^{(i)})\to M_{n}(\FF(x_{h,g}^{(i)}))$
is injective and can be easily seen to be an isomorphism. It follows
that $\dim_{Z(D_{G,k})}(\left(D_{G,k}\right)_{e})=\dim_{\FF(x_{h,g}^{(i)})}(\left(M_{n}(\FF(x_{h,g}^{(i)}))\right)_{e})=n$
is finite, hence $\left(D_{G,k}\right)_{e}$ is indeed a field, and
therefore isomorphic to $\LL$.
\end{proof}
Note that the elements generating the field $\LL$ in the last theorem
correspond to ``cycles'' on a Cayley graph for the group $G$. In
\secref{Flows-in-Graphs} this notion will be formalized and generalized
to other graphs, but first we recall some definition and results required
for this study.

\section{\label{sec:Preliminaries}$G$-lattices and field invariants}

Fix a finite group $G$ and a field $\FF$ of characteristic zero
with a $G$-action (possibly trivial). In this section we recall the
basic definitions and results regarding $G$-lattices and their corresponding
field invariants needed for this paper. Further details and proofs
can be found in \cite{lorenz_multiplicative_2005} and \cite{le_bruyn_permutation_????}.
\begin{defn}
Let $\LL/\FF$ be a finitely generated field extension. Then:
\begin{itemize}
\item $\LL$ is called a \emph{rational extension} of $\FF$ if $\LL=\FF\left(x_{1},...,x_{n}\right)$
for some algebraically independent indeterminates $\left\{ x_{1},...,x_{n}\right\} $.
\item $\LL$ is called a \emph{stably rational extension} of $\FF$ if $\LL(y_{1},...,y_{m})$
is rational over $\FF$ for some $\left\{ y_{1},...,y_{m}\right\} $
algebraically independent over $\LL$.
\item $\LL$ is called \emph{retract rational extension} of $\FF$ if $\LL$
is the fraction field of some $\FF$-algebra $A$, and there are homomorphisms
$\iota:A\to\FF\left[x_{1},...,x_{n}\right]\left[s^{-1}\right]$ and
$\pi:\FF\left[x_{1},...,x_{n}\right]\left[s^{-1}\right]\to A$ such
that $\pi\circ\iota=id_{A}$.
\item $\LL$ is called a \emph{unirational extension} of $\FF$ if $\LL$
can be embedded in a rational extension of $\FF$.
\end{itemize}
\end{defn}
We always have 
\[
\mbox{rational}\subseteq\mbox{stably rational}\subseteq\mbox{retract rational}\subseteq\mbox{unirational}.
\]

We remark here that the inclusions above are all proper, although
the examples are not trivial.
\begin{defn}
Two finitely generated field extensions $\LL_{1},\LL_{2}/\FF$ are
called \emph{stably isomorphic} (over $\FF$) if there are $x_{1},...,x_{m}$
and $y_{1},...,y_{k}$ algebraically independent over $\LL_{1},\LL_{2}$
respectively such that $\LL_{1}\left(x_{1},...,x_{m}\right)\cong\LL_{2}(y_{1},...,y_{k})$
as $\FF$-algebras. In this case we write $\LL_{1}\approx_{\FF}\LL_{2}$
or just $\LL_{1}\approx\LL_{2}$ if there is no ambiguity.
\end{defn}
A \emph{$G$-lattice} M is a $G$-module which is finitely generated
as an abelian group, namely $M\cong\ZZ^{n}$ for some $n\in\NN$.
We denote by $\FF\left[M\right]$ the group algebra of $M$ which
is isomorphic to the algebra of Laurent polynomials in $rank(M)=n$
variables. The $G$-action on $M$ and $\FF$ induce a $G$-action
on $\FF\left[M\right]$ by $g(\alpha x^{m})=g(\alpha)x^{g(m)}$ for
any $g\in G$, $\alpha\in\FF$ and $m\in M$. This action is extended
naturally to the field of fractions $\FF\left(M\right)$ of $\FF\left[M\right]$.
The main question is to find out how close is the extension $\FF(M)^{G}/\FF^{G}$
to being rational.

A $G$-lattice $M$ is called a \emph{permutation lattice} if it has
a $G$-stable $\ZZ$ basis $X$, in which case we write $M=\ZZ X$.
Since each $G$-set $X$ is a disjoint union of sets of the form $\nicefrac{G}{H}$
for some subgroup $H\leq G$, it follows that $\ZZ X\cong\bigoplus_{1}^{t}\ZZ\nicefrac{G}{H_{i}}$.
Notice that if $X=\left\{ x_{1},...,x_{n}\right\} $, then $\FF\left(\ZZ X\right)=\FF\left(y_{x_{1}},...,y_{x_{n}}\right)$
where the $y_{x_{i}}$ are algebraically independent over $\FF$ and
$g(y_{x})=y_{g(x)}$ for any $x\in X$. 

The importance of permutation lattices is given in the next two results.
\begin{prop}[Masuda]
\label{prop:Masuda} Let $\nicefrac{\FF}{\FF^{G}}$ be a $G$-Galois
extension of fields and $P$ a permutation lattice. Then $\FF(P)^{G}$
is a rational extension of $\FF^{G}$.\end{prop}
\begin{thm}
\label{thm:reduction_by_permutation}Let $\xymatrix{0\ar[r] & M_{1}\ar[r] & M_{2}\ar[r] & P\ar[r] & 0}
$ be an exact sequence of $G$-lattices where $P$ is a permutation
lattice. Then $\FF(M_{2})^{G}$ is rational over $\FF(M_{1})^{G}$
if one of the following conditions is satisfied:
\begin{enumerate}
\item $M_{1}$ is a faithful $G$-lattice and $\FF$ has a trivial $G$-action. 
\item $\FF/\FF^{G}$ is a $G$-Galois extension of fields.
\end{enumerate}
\end{thm}
\begin{rem}
\label{rem:Noether-Problem}Noether's Problem asks if given a field
$\FF$ with trivial $G$-action and a faithful permutation lattice
$M$, how close is $\FF(M)^{G}/\FF$ to be a rational extension. This
problem was studied extensively in the literature and was solved for
many families of groups. I particular, a positive solution for this
problem implies a positive solution to the inverse Galois problem.
For example, the rationality of $\FF\left(\ZZ\nicefrac{S_{n}}{S_{n-1}}\right)^{S_{n}}/\FF$
follows from the fundamental theorem of symmetric polynomials and
was used by Hilbert to show that $S_{n}$ is realizable as a Galois
extension over $\QQ$. Noether's problem was also solved completely
for abelian groups by Lenstra in \cite{lenstra_rational_1974} for
any field $\FF$. For the case where $\FF$ contains a root of unity
of the order of the exponent of $G$ and $G$ is abelian, Fischer
showed that $\FF(\ZZ G)^{G}/\FF$ is always rational \cite{fischer_isomorphie_1915}.
A generalization of this result where the action of $G$ on $\FF$
is not trivial is given in \lemref{Fischer-generalize}. Many cases
appearing in this paper are answered by a reduction to Noether's Problem.
\end{rem}
A $G$-lattice $M$ is called \emph{quasi-permutation }if there is
an exact sequence $\xymatrix{0\ar[r] & M\ar[r] & Q\ar[r] & P\ar[r] & 0}
$ where $P$ and $Q$ are permutation lattices. In particular, under
the conditions of the last theorem on $M$ and $\FF$, we get that
$\FF(M)^{G}\approx\FF(\ZZ G)^{G}$.

One trivial way to find an exact sequence as in the theorem above
is if $M_{2}\cong M_{1}\oplus P$. With this in mind we say that two
lattice $M_{1},M_{2}$ are called \emph{permutation equivalent }or
just \emph{equivalent} if there are permutation lattices $P_{1},P_{2}$
such that $M_{1}\oplus P_{1}\cong M_{2}\oplus P_{2}$. In this case
we write $M_{1}\sim M_{2}$ and denote the equivalence class by $\left[M_{1}\right]$.
It follows, for example, that if $M_{1},M_{2}$ are faithful $G$-lattices,
$\FF$ has a trivial $G$-action and $M_{1}\sim M_{2}$, then $\FF(M_{1})^{G}\approx\FF(M_{2})^{G}$. 

The set of $G$-lattices modulo the relation $\sim$ has the structure
of an abelian semi group with the action $\left[M_{1}\right]+\left[M_{2}\right]=\left[M_{1}\oplus M_{2}\right]$.
A lattice $M$ is called \emph{invertible }or \emph{permutation projective}
if $\left[M\right]$ is invertible in this semigroup, or equivalently
$M$ is a direct summand of a permutation lattice.

For a subgroup $\tilde{G}\leq G$, we denote by $\hat{H}^{i}(\tilde{G},M)$
the $i$-th Tate cohomology group. Recall that any short exact sequence
$0\to A\to B\to C\to0$ defines a long exact series 
\[
\xymatrix{\cdots\ar[r] & \hat{H}^{-1}(\tilde{G},C)\ar[r] & \hat{H}^{0}(\tilde{G},A)\ar[r] & \hat{H}^{0}(\tilde{G},B)\ar[r] & \hat{H}^{0}(\tilde{G},C)\ar[r] & \hat{H}^{1}(\tilde{G},A)\ar[r] & \cdots}
.
\]

A lattice $M$ is called \emph{flasque} or \emph{$\hat{H}^{-1}$-trivial
}(resp. \emph{coflasque} or \emph{$\hat{H}^{1}$-trivial })\emph{
if} $\hat{H}^{-1}(\tilde{G},M)=0$ for any subgroup $\tilde{G}\leq G$
(resp. $\hat{H}^{1}(\tilde{G},M)=0$). Any permutation lattice, and
therefore any invertible lattice, is flasque and coflasque, though
the converse is not necessarily true. 

It is well known that any surjection $M\to P$ on a projective module
$P$ is split. The following is the analog for invertible modules.
\begin{lem}
\label{lem:splitting_invertibles}Let 
\[
\xymatrix{0\ar[r] & C\ar[r] & M_{1}\ar[r] & I\ar[r] & 0\\
0\ar[r] & I\ar[r] & M_{2}\ar[r] & F\ar[r] & 0
}
\]
be exact sequences where $I$ is invertible, $C$ is coflasque and
$F$ is flasque. Then these sequences split.
\end{lem}
A \emph{flasque resolution} of $M$ is a short exact sequence of the
form
\[
\xymatrix{0\ar[r] & M\ar[r] & P\ar[r] & F\ar[r] & 0}
\]
where $P$ is a permutation lattice and $F$ is flasque. Similarly
a \emph{coflasque resolution of $M$ }is a short exact sequence of
the form
\[
\xymatrix{0\ar[r] & C\ar[r] & P\ar[r] & M\ar[r] & 0}
\]
where $P$ is a permutation lattice and $C$ is coflasque. Note in
particular that a flasque and coflasque resolutions of an invertible
lattice split.

Denote by $M^{*}$ the \emph{dual of $M$}, namely $M^{*}=Hom(M,\ZZ)$.
It has a natural $G$-structure making it into a $G$-lattice. It
follows from the Tate duality that $\hat{H}^{i}(\tilde{G},M)\cong\hat{H}^{-i}(\tilde{G},M^{*})$
for any lattice $M$, subgroup $\tilde{G}\leq G$ and an integer $i\in\ZZ$,
so in particular $M$ is flasque if and only if $M^{*}$ is coflasque.
Since any permutation lattice is self dual, it follows that the dual
of a flasque resolution of $M$ is a coflasque resolution of $M^{*}$
and vice versa.

It is well known that any lattice $M$ admits a coflasque (and therefore
a flasque) resolution. Moreover, the lattices $F$ and $C$ appearing
in the sequences above are unique up to the relation $\sim$ (which
is the analog of Schanuel's lemma).
\begin{lem}
Let $M_{1},M_{2}$ be lattices and let $F_{1},F_{2}$ be flasque lattices
appearing in flasque resolutions of $M_{1}$ and $M_{2}$ respectively.
If $M_{1}\sim M_{2}$, then $F_{1}\sim F_{2}$. 
\end{lem}
Following this theorem, we write $\left[M\right]^{fl}$ to be the
equivalence class $\left[F\right]$ where $F$ is a flasque lattice
appearing in a flasque resolution of $M$. Similarly we write $\left[M\right]^{cofl}$
for coflasque resolutions.

The next result generalizes \thmref{reduction_by_permutation}.
\begin{thm}
\label{thm:reduction_by_flasque}Let $M_{1},M_{2}$ be $G$-lattices
such that $\left[M_{1}\right]^{fl}=\left[M_{2}\right]^{fl}$. Then
$\FF(M_{2})^{G}\approx\FF(M_{2})^{G}$ if one of the following conditions
is satisfied:
\begin{enumerate}
\item $M_{1}$ and $M_{2}$ are faithful $G$-lattices and $\FF$ has a
trivial $G$-action. 
\item $\FF/\FF^{G}$ is a $G$-Galois extension of fields.
\end{enumerate}
\end{thm}

\section{\label{sec:Flows-in-Graphs}Flows in Graphs}

For the rest of this paper, we fix a field $\FF$ of characteristic
zero with trivial $G$-action.

Let $X=(V,E)$ be a (nonempty) finite graph (not necessarily simple).
For a directed edge $e\in E$ we denote by $e_{S},e_{T}$ the source
and target of the edge. A flow on $X$ is a function $f:E\to\ZZ$
such that 
\[
\sum_{\substack{e\in E\\
e_{S}=v
}
}f(e)=\sum_{\substack{e\in E\\
e_{T}=v
}
}f(e)\quad\mbox{for all }v\in V.
\]
The set of all flows is denoted by $Fl(X)$ or $Fl(V,E)$, which is
of course an abelian group under the addition of functions. We will
usually consider only connected graphs, where connected means that
the underlying undirected graph is connected.

Another well known way to define this group comes from the homology
groups of $X$. Let $\ZZ^{E}$, $\ZZ^{V}$ be the set of all integer
valued functions on $E$, $V$ respectively and consider the following
exact sequence
\begin{equation}
\xymatrix{\ZZ^{E}\ar[r]^{\partial_{E}} & \ZZ^{V}\ar[r]^{\varepsilon_{V}} & \ZZ\ar[r] & 0}
\end{equation}
where 
\begin{eqnarray*}
\mbox{for all }f\in\ZZ^{E} & : & \partial_{E}(f)(v)={\displaystyle \sum_{e_{T}=v}}f(e)-\sum_{e_{S}=v}f(e)\\
\mbox{for all }h\in\ZZ^{V} & : & \varepsilon_{V}(h)=\sum_{v\in V}h(v).
\end{eqnarray*}
Let $I_{V}=\ker(\varepsilon_{V})$. It is easily seen that $Im(\partial_{E})\subseteq I_{V}$
and the equality follows from connectedness of the graph. Finally,
the group $Fl(X)$ is nothing more than $\ker(\partial_{E})$. Since
$Fl(X)$ is a subgroup of the f.g. free abelian group $\ZZ^{E}$,
it is a f.g. free abelian group, and from the exact sequence above
it has rank $\left|E\right|-\left|V\right|+1$.

Next we consider group actions on graphs.
\begin{defn}
Let $G$ be a finite group. We say that a $G$-action on $V,E$ is
a graph action if the action on $E$ and $V$ is compatible, namely
for every $g\in G$ and $e\in E$ we have that $g(e_{S})=g(e)_{S}$
and $g(e_{T})=g(e)_{T}$.
\end{defn}
The action of $G$ on $V,E$ induce a $G$-action on $\ZZ^{V},\ZZ^{E}$
respectively. Letting $\ZZ$ have a trivial $G$-action, we get that
the sequence above is a sequence of $G$-modules. The induced action
on $Fl(X)=\ker(\partial_{E})$ is the natural one, taking a flow $f$
and acting on it by the automorphisms of the graph. Finally, note
that there is a natural isomorphism between $\ZZ V,\;\ZZ E$ and their
duals $\ZZ^{V},\;\ZZ^{E}$ respectively. We shall use these two presentations
interchangeably.\\

The aim of this paper is to study the fields $\FF(Fl(X)\oplus\ZZ G)^{G}$
and to determine how close they are to being purely rational extensions
of $\FF$. The addition of $\ZZ G$ to $Fl(X)$ is intended to make
the lattice into a faithful $G$-lattice. If the action of $G$ on
the $Fl(X)$ is already faithful, then $\FF(Fl(X)\oplus\ZZ G)^{G}$
is rational over $\FF(Fl(X))^{G}$ by \thmref{reduction_by_permutation},
thus in many cases this field is considered instead. In case the $G$-action
on $Fl(X)$ has nonzero kernel $K$, it is interesting to ask whether
there is a connection between $\FF(Fl(X)\oplus\ZZ G)^{G}$ and $\FF(Fl(X))^{G/K}$
(see for example \subref{Semidirect-product}).

We note that while the entire investigation can be made using only
the $G$-lattice notation without any reference to the flows in the
graph, we have found that the use of flows have made some of the proofs
much more intuitive, and in particular it helps in the later parts
when we look for certain nice bases for the lattices $Fl(X)$. The
main idea is that if $f\in Fl(V,E)$ is any flow such that $f(e)=1$
for some $e\in E$ and $g$ is any other flow, then $g-g(e)f$ is
again a flow which is supported on $(V,E\backslash\left\{ e\right\} )$
(we redirected the flow on $e$ using $f$). Thus $Fl(V,E)\cong\ZZ f\oplus Fl(V,E\backslash\left\{ e\right\} )$
as abelian group. The next lemma generalizes this idea.
\begin{lem}
\label{lem:basis_for_flows}Let $X=(V,E)$ be a connected graph, $T=(V,E')$
a subgraph such that its underlying undirected graph is a tree, and
write $E\backslash E'=\left\{ e_{1},...,e_{r}\right\} $. Let $f_{1},...,f_{r}\in Fl(V,E)$
such that the matrix $(f_{i}(e_{j}))$ is upper triangular with $\pm1$
on the diagonal. Then $\left\{ f_{1},...,f_{r}\right\} $ is a basis
for $Fl(V,E)$.\end{lem}
\begin{proof}
The proof is left to the reader.\end{proof}
\begin{rem}
Actually, the lemma above still holds whenever the matrix $\left(f_{i}(e_{j})\right)$
is invertible, but we will only encounter triangular matrices in this
paper.
\end{rem}
The most natural choice of a set of edges is the complete directed
graph. Let $E_{-}=\left\{ \left(u,v\right)\;\mid\; u,v\in V\;\mbox{distinct}\right\} $
and set $E_{+}=E_{-}\cup\left\{ \left(v,v\right)\;\mid\; v\in V\right\} $.
Clearly every flow in $Fl(V,E_{+})$ can be decomposed as a flow supported
on $E_{-}$ and a flow supported on $E_{+}\backslash E_{-}$ (self
loops), or in other words $Fl(V,E_{+})\cong Fl(V,E_{+})\oplus Fl(V,E_{+}\backslash E_{-})$.
The set $E_{+}\backslash E_{-}$ is just the self loops and $G$ acts
on it as it acts on $V$, hence $Fl(V,E_{+})\cong Fl(V,E_{-})\oplus\ZZ V$.
Thus, it is sufficient to consider the flows on the full directed
graph without loops. The next lemma gives us a more general way to
reduce the number of edges in the graph.
\begin{lem}
Let $X=(V,E),X'=(V,E')$ be two connected graphs with $G$-actions.
Then $\FF(Fl(X)\oplus\ZZ G)^{G}$ is stably isomorphic to $\FF(Fl(X')\oplus\ZZ G)^{G}$
over $\FF$. \end{lem}
\begin{proof}
Consider the following exact diagram:
\[
\xymatrix{ &  & 0\ar[d] & 0\ar[d]\\
 &  & Fl(X')\ar[d]\ar[r]^{=} & Fl(X')\ar[d]\\
0\ar[r] & Fl(X)\ar[r]\ar[d]_{=} & \ZZ E\times_{I_{V}}\ZZ E'\ar[d]\ar[r] & \ZZ E'\ar[d]^{\partial_{E'}}\ar[r] & 0\\
0\ar[r] & Fl(X)\ar[r] & \ZZ E\ar[r]^{\partial_{E}}\ar[d] & I_{V}\ar[r]\ar[d] & 0\\
 &  & 0 & 0
}
\]
where $\ZZ E\times_{I_{V}}\ZZ E'$ is the standard pullback. Since
$\ZZ E,\ZZ E'$ are permutation modules, it follows from \thmref{reduction_by_permutation}
that $\FF(\ZZ G\oplus Fl(X))^{G}\approx_{\FF}\FF(\ZZ G\oplus Fl(X'))^{G}$.
\end{proof}
The previous lemma shows that up to stable rationality, the field
$\FF(Fl(V,E)\oplus\ZZ G)^{G}$ is only a function of the set of vertices
$V$ with its $G$-action. This result can also be motivated from
the nature of flows on graphs. If $(V,E)$ is any connected $G$-graph
and $(V,E')$ is a connected $G$-subgraph, then any flow going through
an edge $e\in E\backslash E'$, can be redirected to a flow in the
graph $(V,E')$, since the vertices in $e$ are already connected
in $E'$. We can thus allow ourselves to write $Fl(V)$ whenever the
exact structure of $E$ is not important up to stable isomorphism
of fields.

Next we reduce the number of vertices in $V$.
\begin{lem}
\label{lem:Removing_vertices}Let $V=\bigsqcup V_{i}$ be the decomposition
to $G$-orbits. Suppose that for some $i\neq j$ there is a $G$-equivariant
map $\psi:V_{i}\to V_{j}$. Then $\FF(Fl(V)\oplus\ZZ G)^{G}\approx\FF(Fl(V')\oplus\ZZ G)^{G}$
where $V'=V\backslash V_{i}$.\end{lem}
\begin{proof}
Choose some $u\in V_{i}$ and set $v=\psi(u)$, then $stab_{G}(u)\subseteq stab_{G}(v)$.
Let $T=\left\{ t_{1},...,t_{k}\right\} $ be a set of left coset representatives
of $stab_{G}(u)$ in $G$. Let $E'=\left\{ (w_{1},w_{2})\;\mid\; w_{1},w_{2}\in V'\right\} $
so that $X'=(V',E')$ is the complete (directed) graph on $V'$ and
set
\[
E=E'\cup\left\{ (t_{i}u,t_{i}v)\;\mid\;1\leq i\leq k\right\} .
\]
It is clear that $E'$ is $G$-stable, and the stability of $E$ follows
from the fact that $stab_{G}(u)\subseteq stab_{G}(v)$. The graph
$X=(V,E)$ is just the full connected graph on $V'$ plus a single
edge from each vertex in $V_{i}$ to some vertex in $V_{j}\subseteq V\backslash V_{i}$.
In particular, any flow in $X$ must be supported on $V\backslash V_{i}$,
so that $Fl(X)=Fl(X')$ and the lemma follows.
\end{proof}
The last lemma shows that we need only to consider such decompositions
$V=\bigsqcup V_{i}$ where $Hom(V_{i},V_{j})=\emptyset$ whenever
$i\neq j$. 

Note that if there is a map $V_{i}\to V_{j}$, then $\left|V_{j}\right|$
divides $\left|V_{i}\right|$. The opposite is also true if we have
that $\left|V_{j}\right|=1$, and thus we have the following corollary.
\begin{cor}
\label{cor:fixed_vertex}If $X=(V,E)$ is a connected $G$-graph which
contains a vertex $v\in V$ with $stab_{G}(v)=G$, then $\FF(Fl(X)\oplus\ZZ G)^{G}\approx\FF(\ZZ G)^{G}$.
\end{cor}
The next generalization is when $1$ is not one of the cardinalities
of the $V_{i}$, but is their greatest common divisor. Assuming that
$gcd\left(\left|V_{1}\right|,...,\left|V_{m}\right|\right)=1$, choose
some $a_{i}\in\ZZ$ such that $\sum_{1}^{m}a_{i}\left|V_{i}\right|=1$.
For each $i$ choose some $v_{i}\in V_{i}$ and let $N_{i}={\displaystyle \sum_{v\in V_{i}}}v\in\ZZ V$
be the norm element in $\ZZ V_{i}$. As before, we have the exact
sequence 
\[
\xymatrix{0\ar[r] & I_{V}\ar[r] & \bigoplus_{1}^{m}\ZZ V_{i}\ar[r]\sp(0.6){\varepsilon_{V}} & \ZZ\ar[r] & 0}
.
\]
Define $\varphi:\ZZ\to\bigoplus_{1}^{m}\ZZ V_{i}$ by $\varphi(k)=\sum_{1}^{m}ka_{i}N_{i}$.
This map is $G$-equivariant since the $N_{i}$ are $G$-invariant,
and since $\varepsilon_{V}(\varphi(k))=k\sum_{1}^{m}a_{i}\left|V_{i}\right|=k$,
the sequence splits.

If $E$ is any set of edges such that $X=(V,E)$ is a connected $G$-graph,
then we have the following exact sequence
\[
\xymatrix{0\ar[r] & Fl(X)\ar[r] & \ZZ\oplus\ZZ E\ar[r]\sp(0.4){id\oplus\partial_{E}} & \ZZ\oplus I_{V}\cong\ZZ V\ar[r] & 0}
.
\]
It follows that $Fl(X)$ is quasi permutation and therefore 
\[
\FF(Fl(X)\oplus\ZZ G)^{G}\approx\FF(\ZZ\oplus\ZZ E\oplus\ZZ G)^{G}\approx\FF(\ZZ G)^{G},
\]
thus reducing the problem to Noether's problem. As mentioned in \remref{Noether-Problem},
the rationality of $\FF(\ZZ G)^{G}/\FF$ is known in many cases (for
example if $G=S_{n}$, $G$ abelian and $\FF$ contains enough roots
of unity, etc).

\section{\label{sec:The-center}The center of the generic $G$-crossed product}

The second natural choice of graph is taking $V=G$, which as explained
in the end of \secref{Generic-Crossed-Products}, corresponds (up
to a stable isomorphism) to the center of the generic $G$-crossed
product.

Let $E_{+}$ be the set of all edges in the full directed graph on
$V$ (including self loops). Fixing an identification of $V$ with
$G$, define a degree function $deg:E_{+}\to G$ by $\deg\left(g,h\right)=g^{-1}h$.
For $S\subseteq G$ let $E_{+}(S)=\deg^{-1}(S)$ be all the edges
of degree in $S$ and write $Cay(G,S)=\left(V,E_{+}(S)\right)$ which
is nothing more than the directed Cayley graph of $G$ corresponding
to the set $S$. In particular $Cay(G,S)$ is connected if and only
if $S$ generates $G$. Excluding the self loops, let $G_{-}=(G\backslash\left\{ e\right\} )$,
and note that $Fl(G,G)\cong Fl(G,G_{-})\oplus\ZZ G$.

As in the general case, if $(V,E),\;(V,E')$ are connected, then $\FF(Fl(V,E))^{G}\sim\FF(Fl(V,E'))^{G}$.
In this case, where $G$ acts freely on $V$, a stronger result holds.
\begin{lem}
\label{lem:Remove_Edges}Suppose that $G$ acts freely on $V$ (so
that $V$ is a disjoint union of copies of $G$). Let $E'\subseteq E$
be two $G$-sets of edges on $V$ such that $(V,E')$ is connected
(and therefore also $(V,E)$). Then $Fl(V,E)\cong Fl(V,E')\oplus\ZZ G^{m}$
where $m=\frac{\left|E\right|-\left|E'\right|}{\left|G\right|}$,
and the isomorphism restricted to $Fl(V,E')$ is the natural embedding.\end{lem}
\begin{proof}
Choose a set of orbit representatives $\left(v_{i},w_{i}\right)$
of $E\backslash E'$. For each $i$, choose some flow in $(V,E')$
from $w_{i}$ to $v_{i}$ (which exists since it is connected) and
complete it with the edge $(v_{i},w_{i})$ to a circular flow $C_{i}$.
By construction, each edge of $E\backslash E'$ appears in exactly
one of the flows $\bigcup_{i}\left\{ gC_{i}\;\mid\; g\in G\right\} $.
These can be used to redirect flows from edges in $E\backslash E'$
into $E'$, or more precisely $Fl(V,E)=Fl(V,E')\oplus\left(\bigoplus_{i}\ZZ\left\{ gC_{i}\;\mid\; g\in G\right\} \right)$
where for each $i$ we have $\ZZ\left\{ gC_{i}\;\mid\; g\in G\right\} \cong\ZZ G$.\end{proof}
\begin{cor}
\label{cor:flows_for_cyclic_groups}Let $G$ be cyclic with generator
$\sigma$. Then for any $S\subseteq G$ with $\sigma\in S$ we have
$Fl(G,S)\cong Fl(G,\left\{ \sigma\right\} )\oplus\ZZ G^{\left|S\right|-1}$
and $Fl(G,\left\{ \sigma\right\} )\cong\ZZ$ with the trivial action.
In particular, $Fl(G,S)$ is a permutation lattice.\end{cor}
\begin{proof}
The result follows since $Cay(G,\left\{ \sigma\right\} )$ is just
a simple cycle and $G$ acts trivially on $Fl(G,\left\{ \sigma\right\} )\cong\ZZ$.\end{proof}
\begin{rem}
\label{rem:cyclic-groups}By \thmref{reduction_by_permutation}, the
field $\FF(\ZZ\oplus\left(\ZZ G\right)^{k})^{G},\; k\geq1$ is rational
over $\FF(\ZZ G)^{G}$. Fischer's theorem (see \remref{Noether-Problem})
states that if $G$ is abelian with exponent $m$ and $\FF$ contains
a primitive $m$-th root of unity, then $\FF(\ZZ G)^{G}$ is rational
over $\FF$. Thus, if $\FF$ contains a primitive $\left|G\right|$-th
root of unity for $G$ cyclic, then $\FF(Fl(V,E)\oplus\ZZ G)^{G}$
is rational over $\FF$. For $\FF=\QQ$, even the extensions $\QQ(\ZZ C_{n})^{C_{n}}/\QQ$
are not always rational, with the first counter example given by Swan
in \cite{swan_invariant_1969} for $n=47$ and the smallest counter
example is $n=8$ \cite{lenstra_rational_1974}.
\end{rem}
As in many other cases, we will try to study the graphs $Cay(G,S)$,
where $\left\langle S\right\rangle =G$, by restricting the $G$-action
to some subgroups. More generally, if $M$ is a $G$ lattice we denote
by $M_{H}$ its restriction to the $H$-action where $H\leq G$. If
$H$ is any subgroup of $G$ and $S_{0}\subseteq S$ is a generating
set for $H$, then $Cay(H,S_{0})$ has a natural embedding in $Cay(G,S)$.
Thus, turning to their respective flow lattices, we see that $Fl(H,S_{0})$
is an $H$-sublattice of $Fl(G,S)_{H}$ (the lattice with the $H$-action).
The passage to this subgraph is given in the following lemma.
\begin{lem}
\label{lem:reducing_to_subgroup}Let $G$ be a group and $H$ a subgroup.
Let $S$ be a generating set for $G$ containing a generating set
$S_{0}$ for $H$. Then $Fl(G,S)\cong Fl(H,S_{0})\oplus P$ as $H$
lattices where $P$ is a free $H$ lattice, and the restriction of
the isomorphism to $Fl(H,S_{0})$ is the natural embedding. \end{lem}
\begin{proof}
This is just a fine tuning of \lemref{Removing_vertices} together
with the previous lemma and is left to the reader.
\end{proof}
In case $M_{1},M_{2}$ are any two $G$-modules, the tensor product
$M_{1}\otimes M_{2}$ is considered over $\ZZ$ with the diagonal
$G$-action, namely $g(m_{1}\otimes m_{2})=gm_{1}\otimes gm_{2}$
for any $g\in G$ and $m_{i}\in M_{i}$, $i=1,2$. If $M_{1}$ is
a $\ZZ G-\ZZ H$ bimodule and $M_{2}$ an $H$-module for some subgroup
$H\leq G$, then we write $M_{1}\otimes_{H}M_{2}$ for the $G$-module
$M_{1}\otimes_{\ZZ H}M_{2}$ where $G$ acts through $M_{1}$. Under
these notations we have the following.
\begin{lem}
\label{lem:Move_to_subgroup}Let $G$ be a group, $H\leq G$ a subgroup
and $M$ a $G$-lattice and set $M_{1}=\ZZ\nicefrac{G}{H}\otimes M$
and $M_{2}=\ZZ G\otimes_{H}M_{H}$. Then the map $\varphi:M_{1}\to M_{2}$
defined by $\varphi(gH\otimes a)=g\otimes g^{-1}a$ is an isomorphism
of $G$-lattices. \end{lem}
\begin{proof}
The proof is straight forward and is left to the reader.
\end{proof}
In the previous case we studied the action of $G$ on $\tilde{V}=\bigsqcup V_{i}$
where $gcd(\left|V_{1}\right|,...,\left|V_{k}\right|)=1$. Letting
$H_{i}$ be the stabilizer of some $v_{i}\in V_{i}$ for each $i$,
we see that $gcd\left(\left[G:H_{1}\right],...,\left[G:H_{k}\right]\right)=1$,
since $V_{i}\cong\nicefrac{G}{H_{i}}$ as $G$-sets and hence $\ZZ V_{i}\cong\ZZ\nicefrac{G}{H_{i}}$.
Clearly, if $M$ is any invertible $G$-lattice, namely a direct summand
of a permutation $G$-lattice, it is also invertible as an $H$ lattice
for any subgroup $H\leq G$. For the other direction, it is well known
that if $M_{H_{i}}$ is invertible over subgroups $H_{i}\leq G$ that
satisfy the condition above, then $M$ is invertible as a $G$-lattice.
Let us recall the proof.

Consider the split exact sequence 
\[
\xymatrix{0\ar[r] & I_{\tilde{V}}\ar[r] & \bigoplus_{1}^{k}\ZZ V_{i}\ar[r]\sp(0.6){\varepsilon_{\tilde{V}}} & \ZZ\ar[r] & 0}
.
\]
Tensoring (over $\ZZ$) this sequence with the lattice $M$ we get
a split exact sequence
\[
\xymatrix{0\ar[r] & I_{\tilde{V}}\otimes M\ar[r] & \bigoplus_{1}^{k}\left(\ZZ\nicefrac{G}{H_{i}}\otimes M\right)\ar[r]\sp(0.7){\varepsilon_{\tilde{V}}} & M\ar[r] & 0}
,
\]
so that $M$ is a direct summand of $\bigoplus_{1}^{k}\left(\ZZ\nicefrac{G}{H_{i}}\otimes_{\ZZ}M\right)$.
By the previous lemma $\ZZ\nicefrac{G}{H}\otimes M\cong\ZZ G\otimes_{H}M_{H}$,
so the middle term in the sequence above is just $\bigoplus_{1}^{k}\left(\ZZ G\otimes_{H_{i}}M_{H_{i}}\right)$.
It is now clear that if $M_{H_{i}}$ is a direct summand of a permutation
$H_{i}$-lattice, then $M$ is a direct summand of a permutation $G$-lattice.
By \lemref{Remove_Edges}, the invertibility of $Fl(G,S)$ doesn't
depend on the generating set $S$, hence we can assume that we deal
with $Fl(G,G)$. 
\begin{cor}
\label{cor:invertible_by_subgroup}Let $G$ be a group and $H_{i}\leq G$
subgroups such that $gcd\left(\left[G:H_{1}\right],...,\left[G:H_{k}\right]\right)=1$.
The $G$-lattice $Fl(G,G)$ is invertible if and only $Fl(H_{i},H_{i})$
is invertible for each $i$.
\end{cor}
A natural choice for $H_{i}$ in the above corollary are Sylow subgroups
of $G$, and a trivial reason for $Fl(H_{i},H_{i})$ to be invertible
is if $H_{i}$ is cyclic. We are thus led to investigate groups such
that all of their Sylow subgroups are cyclic, which are called $Z$-groups.
These type of groups arise naturally in the theory of lattice invariant
as we shall now recall.

Let $E$ be any set of edges on $V=G$ such that $X=(V,E)$ is connected.
Recall that we have the exact sequences
\[
\xymatrix{0\ar[r] & I_{V}\ar[r] & \ZZ V\ar[r]^{\epsilon} & \ZZ\ar[r] & 0\\
0\ar[r] & Fl(X)\ar[r] & \ZZ E\ar[r]^{\partial} & I_{V}\ar[r] & 0.
}
\]
Since both $\ZZ V$ and $\ZZ E$ are free $G$ modules, it follows
that they are also free over any subgroup $\tilde{G}\leq G$, and
therefore $\hat{H}^{i}(\tilde{G},\ZZ V)=\hat{H}^{i}(\tilde{G},\ZZ E)=0$
for any $i\in\ZZ$. Taking the long exact sequence of the Tate cohomology
we conclude that
\[
\hat{H}^{1}(\tilde{G},Fl(X))\cong\hat{H}^{0}(\tilde{G},I_{V})\cong\hat{H}^{-1}(\tilde{G},\ZZ)=0,
\]
so that $Fl(X)$ is coflasque. In particular, the second sequence
above is a coflasque resolution of $I_{V}$. If $X'=(V,E')$ is any
other connected graph, it follows that $\left[Fl(X')\right]=\left[I_{V}\right]^{cofl}=\left[Fl(X)\right]$,
thus there are permutation lattices $P,P'$ such that $Fl(X)\oplus P\cong Fl(X')\oplus P'$.

Taking $\left[I_{G}\right]^{cofl}=\left[I_{V}\right]^{cofl}=\left[Fl(X)\right]$
and dualizing, we get that $\left[I_{G}^{*}\right]^{fl}=\left[I_{V}^{*}\right]^{fl}=\left[Fl(X)\right]$.
In \cite{endo_classification_1975}, Endo and Miyata proved that there
is an equality of sets $coflasque=flasque=invertible$ for $G$-lattices
if and only if $\left[I_{G}^{*}\right]^{fl}$ is invertible, which
is of course equivalent to $Fl(X)$ being invertible. Another equivalent
condition given by Endo and Miyata is that $G$ is a $Z$-group, namely
that every Sylow subgroup of $G$ is cyclic. 

Groups with cyclic Sylow subgroup also appear in Saltman's representing
objects for crossed products in \cite{saltman_retract_1984}. More
precisely, let $\LL/\LL^{G}$ be a $G$-Galois extension and consider
$Fl(G,G_{-})$. Saltman used Brauer theory in order to prove that
the field $\LL(Fl(G,G_{-}))^{G}$ (which is the center of Saltman's
generic $G$-crossed product for the $\LL/\LL^{G}$ extension) is
a retract rational extension of $\LL^{G}$ if and only if $G$ is
a $Z$-group if and only if there is an exact sequence 
\[
\xymatrix{0\ar[r] & Fl(G,G_{-})\ar[r] & P\ar[r] & I\ar[r] & 0}
\]
where $P$ is a permutation lattice and $I$ is invertible. Since
$Fl(G,G_{-})$ is coflasque, this last condition is also equivalent
to $Fl(G,G_{-})$ being invertible using \lemref{splitting_invertibles}. 

On the other hand, suppose that $G$ is ``far'' from being a $Z$-group
in the sense that it has some Sylow subgroup $P$ which is not a rank
1 or 2 abelian group. S Then Saltman proved that $\FF(Fl(X))^{G}$
is not retract rational over $\FF$ and in particular it is not stably
rational (Theorem 12.17 in \cite{saltman_lectures_1999}).

We note that here we consider the extension $\FF(Fl(X))^{G}/\FF$
where $G$ acts on $\FF$ trivially and not $\LL(Fl(X))^{G}/\LL^{G}$
where $\LL/\LL^{G}$ is a $G$-Galois extension, so we cannot use
directly the invertibility of $Fl(X)$ in case $G$ is a $Z$-group.
Nevertheless, we begin by investigating $Z$-groups.

\subsection{\label{sub:Groups-with-cyclic-sylow}Groups with cyclic $p$-Sylow
subgroups}

It is well known that each $Z$-group $G$ is isomorphic to a semidirect
product $C_{n}\rtimes C_{m}$ where $C_{n}=\left\langle \sigma\right\rangle ,C_{m}=\left\langle \tau\right\rangle $
are cyclic of coprime orders $n$ and $m$ respectively (see \cite{hall_theory_1999},
section 9.4). As already seen, if $G$ is a $Z$-group, then $Fl(G,E)$
is invertible. It follows that if $\LL/\LL^{G}$ is $G$-Galois, then
$\LL\left(M\right)^{G}/\LL^{G}$ is a retract rational extension.
A similar result holds in our case.
\begin{lem}
\label{lem:Z-group_is_retract}Let $G$ be a Z-group. Let $\zeta$
be a primitive $t$-th root of unity, where $t$ is the highest power
of $2$ dividing $\left|G\right|$. If $Gal(\FF(\zeta)/\FF)$ is cyclic,
then the field $\FF(M)^{G}$ is retract rational over $\FF$ for any
invertible faithful $G$-lattice $M$. \end{lem}
\begin{proof}
Given any invertible lattice $M$ and a $G$-Galois extension $\LL/\LL^{G}$,
Saltman showed that $\LL(M)^{G}/\LL^{G}$ is a retract rational extension
(\cite{saltman_retract_1984}, Theorem 3.14). In addition, if $Gal(\FF(\zeta)/\FF)$
is cyclic and $G$ is a $Z$-group, he proved that for $\LL=\FF(\ZZ G)$
the extension $\LL^{G}/\FF$ is also retract rational (see \cite{saltman_generic_1982}
,Theorems 2.1, 3.5, 5.3 where he proves that for such groups there
is a generic Galois extension and has a lifting property, and in \cite{saltman_retract_1984}
Theorem 3.12 he proves that the lifting property implies retract rationality).
In particular we get that $\FF\subseteq\FF(\ZZ G)^{G}\subseteq\FF(\ZZ G\oplus M)^{G}$
is a tower of retract rational extensions, which by \cite{kang_retract_2012}
means that $\FF(\ZZ G\oplus M)^{G}/\FF$ is also a retract rational
extension. Since $M$ is faithful and $\ZZ G$ is permutation, \propref{Masuda}
shows that $\FF(\ZZ G\oplus M)^{G}/\FF(M)^{G}$ is rational and it
follows that $\FF(M)^{G}/\FF$ is retract rational.
\end{proof}
In particular, since $Fl(G,E)$ is invertible, we conclude that$\FF(Fl(G,E)\oplus\ZZ G)^{G}/\FF$
is retract rational, thus proving part (1) of \thmref{Main_Theorem_sylow}.

Now that we know that $\FF(Fl(G))^{G}/\FF$ is a retract rational
extension given that $\FF$ contains enough roots of unity, we turn
to ask if it is a stably rational extension. One such result was given
by Snider in \cite{snider_is_1979}. Snider proved that if $G=\left\langle \sigma,\tau\;\mid\;\sigma^{n}=\tau^{2}=e,\;\tau\sigma\tau^{-1}=\sigma^{-1}\right\rangle $
is a dihedral group with $n$ odd and $\FF$ contains a primitive
$n$-th root of unity, then $\FF(Fl(G))^{G}/\FF$ is stably rational. 

As mentioned above, each $Z$-group $G$ is isomorphic to a semidirect
product of cyclic groups of coprime orders $C_{n}\rtimes C_{m}$.
Note that $C_{m}$ acts on $C_{n}$ by conjugation, and in the dihedral
case the action is faithful. In what follows, we shall prove that
we can always reduce the question to faithful actions.

Unless stated otherwise, for the rest of this section we fix the following
notations.

Let $G=\left\langle \sigma,\tau\;\mid\;\sigma^{n}=\tau^{m}=e,\;\tau^{-1}\sigma\tau=\sigma^{r}\right\rangle $
where $(n,m)=1$ and $r\in\ZZ$ satisfying $r^{m}\equiv_{n}1$. Let
$\left\langle \tau\right\rangle $ act on $\left\langle \sigma\right\rangle $
by conjugation and let $K=Z(G)\cap\left\langle \tau\right\rangle $
be the kernel of this action. Note that if $m_{0}$ is the order of
$r$ in $\nicefrac{\ZZ}{n\ZZ}$, then $K=\left\langle \tau^{m_{0}}\right\rangle $.
It follows that $\tilde{G}\cong\nicefrac{G}{K}=\left\langle \tilde{\sigma},\tilde{\tau}\;\mid\;\tilde{\sigma}^{n}=\tilde{\tau}^{m_{0}}=e,\;\tilde{\tau}\tilde{\sigma}\tilde{\tau}^{-1}=\tilde{\sigma}^{r}\right\rangle $
and $\left\langle \tilde{\tau}\right\rangle $ acts faithfully on
$\left\langle \tilde{\sigma}\right\rangle $. We also assume that
$\sigma^{r}\neq\sigma$, and in particular $n>2$, since otherwise
$G$ is cyclic. Finally, we let $M=Fl(G,\left\{ \sigma,\tau\right\} )$.
Our next goal is to show that we can move from $G$ to $\tilde{G}$
when computing the invariants.

As we have seen before, we have the split exact sequence
\[
\xymatrix{0\ar[r] & I\otimes M\ar[r] & \left(\ZZ\nicefrac{G}{\left\langle \sigma\right\rangle }\otimes M\right)\oplus\left(\ZZ\nicefrac{G}{\left\langle \tau\right\rangle }\otimes M\right)\ar[r]\sp(0.75)\varepsilon & M\ar[r] & 0}
.
\]
Let us investigate the projection $\varepsilon$ a little further.

From \lemref{reducing_to_subgroup} and \ref{lem:Move_to_subgroup}
we have the isomorphism
\[
\ZZ\nicefrac{G}{\left\langle \sigma\right\rangle }\otimes M\cong\ZZ G\otimes_{\left\langle \sigma\right\rangle }M_{\left\langle \sigma\right\rangle }\cong\ZZ G\otimes_{\left\langle \sigma\right\rangle }(P_{1}\oplus Fl(\left\langle \sigma\right\rangle ,\left\langle \sigma\right\rangle ))\cong\ZZ G\otimes_{\left\langle \sigma\right\rangle }\left(P_{1}\oplus P_{2}\oplus\ZZ\right)
\]
where $P_{1},P_{2}$ are free $\left\langle \sigma\right\rangle $-lattices.
Let $\hat{S}=e\otimes1$ in the right most term where $1\in\ZZ$.
The first and second isomorphisms from the right send it to $e\otimes S$
where $S$ is the flow on the cycle $e\to\sigma\to\sigma^{2}\to\cdots\to\sigma^{n-1}\to e$
which is sent by the last isomorphism to $\left\langle \sigma\right\rangle \otimes S$. 

Since $P_{1},P_{2}$ are free $\left\langle \sigma\right\rangle $-lattices,
we get that $\ZZ G\otimes_{\left\langle \sigma\right\rangle }\left(P_{1}\oplus P_{2}\oplus\ZZ\right)\cong Q\oplus\ZZ\nicefrac{G}{\left\langle \sigma\right\rangle }$
and $\hat{S}$ from above corresponds to $e\left\langle \sigma\right\rangle $
in $\ZZ\nicefrac{G}{\left\langle \sigma\right\rangle }$. By the previous
argument we have that $\varepsilon(\hat{S})=S$. A similar argument
is true for the subgroup $\left\langle \tau\right\rangle $. Putting
it all together we get the split exact sequence
\[
\xymatrix{0\ar[r] & I\otimes M\ar[r] & P\oplus\ZZ\nicefrac{G}{\left\langle \sigma\right\rangle }\oplus\ZZ\nicefrac{G}{\left\langle \tau\right\rangle }\ar[r]\sp(0.7)\varepsilon & M\ar[r] & 0}
\]
where $P$ is a $G$-free lattice. Additionally, denoting by $\hat{S},\hat{T}$
the elements $e\left\langle \sigma\right\rangle ,\; e\left\langle \tau\right\rangle $
in $\ZZ\nicefrac{G}{\left\langle \sigma\right\rangle },\ZZ\nicefrac{G}{\left\langle \tau\right\rangle }$
respectively and letting $S,T$ be the flows $e\to\sigma\to\cdots\to\sigma^{n-1}\to e$
and $e\to\tau\to\cdots\to\tau^{m-1}\to e$ respectively, we get that
$\varepsilon(\hat{S})=S$ and $\varepsilon(\hat{T})=T$. Our next
step is to find another such exact sequence which is easier to work
with.

Consider an exact sequence
\[
\xymatrix{0\ar[r] & \ker(\pi)\ar[r] & Q\oplus\ZZ\nicefrac{G}{\left\langle \sigma\right\rangle }\oplus\ZZ\nicefrac{G}{\left\langle \tau\right\rangle }\ar[r]\sp(0.7)\pi & M\ar[r] & 0}
\]
where the restriction of $\pi$ to $\ZZ\nicefrac{G}{\left\langle \sigma\right\rangle }\oplus\ZZ\nicefrac{G}{\left\langle \tau\right\rangle }$
coincides with $\varepsilon$, and $Q$ is some free $G$-lattice.
If $P'$ is any free $G$-lattice, $H$ is any subgroup of $G$ and
$i\in\ZZ$, then
\[
\hat{H}^{i}(P'\oplus\ZZ\nicefrac{G}{\left\langle \sigma\right\rangle }\oplus\ZZ\nicefrac{G}{\left\langle \tau\right\rangle })\cong\hat{H}^{i}(\ZZ\nicefrac{G}{\left\langle \sigma\right\rangle }\oplus\ZZ\nicefrac{G}{\left\langle \tau\right\rangle }),
\]
and since $\pi,\varepsilon$ coincide on $\ZZ\nicefrac{G}{\left\langle \sigma\right\rangle }\oplus\ZZ\nicefrac{G}{\left\langle \tau\right\rangle }$,
it follows that their induced map on the Tate cohomology groups also
coincide. The homomorphism $\varepsilon$ splits so its induced map
in cohomology is always surjective, and therefore also $\pi$'s induced
maps. Finally, we have that 
\[
\xymatrix{\hat{H}^{0}(H,P'\oplus\ZZ\nicefrac{G}{\left\langle \sigma\right\rangle }\oplus\ZZ\nicefrac{G}{\left\langle \tau\right\rangle })\ar[r] & \hat{H}^{0}(H,M)\ar[r] & \hat{H}^{1}(H,\ker(\pi))\ar[r] & \hat{H}^{1}(P'\oplus\ZZ\nicefrac{G}{\left\langle \sigma\right\rangle }\oplus\ZZ\nicefrac{G}{\left\langle \tau\right\rangle })=0}
\]
so that $\hat{H}^{1}(H,\ker(\pi))=0$ for all subgroups $H\leq G$,
or equivalently it is a coflasque lattice. Since $M$ is invertible
and $\ker(\pi)$ coflasque, we conclude from \lemref{splitting_invertibles}
that $\ker(\pi)\oplus M\cong\ZZ G\oplus\ZZ\nicefrac{G}{\left\langle \sigma\right\rangle }\oplus\ZZ\nicefrac{G}{\left\langle \tau\right\rangle }$
so that $\ker(\pi)$ is invertible and therefore $\left[M\right]^{fl}=\left[M\right]^{cofl}=\left[\ker(\pi)\right]$.

Let us consider the map $\pi$ defined as follows. Recall that $1\leq r\leq n-1$
is the integer such that $\sigma\tau=\tau\sigma^{r}$. Define the
flows $S,T,A$ to be
\begin{eqnarray*}
S & := & e\to\sigma\to\cdots\to\sigma^{n-1}\to e\\
T & := & e\to\tau\to\cdots\to\tau^{m-1}\to e\\
A & := & (e\to\sigma\to\sigma\tau)-(e\to\tau\to\tau\sigma\to\cdots\to\tau\sigma^{r}).
\end{eqnarray*}
Clearly we have that $\sigma(S)=S$ and $\tau(T)=T$. Let $\hat{S}=e\left\langle \sigma\right\rangle ,\hat{T}=e\left\langle \tau\right\rangle $
and $\hat{A}=e$ be generators for $\ZZ\nicefrac{G}{\left\langle \sigma\right\rangle },\ZZ\nicefrac{G}{\left\langle t\right\rangle }$
and $\ZZ G$ respectively (as $\ZZ G$-modules), and define $\pi:\ZZ\nicefrac{G}{\left\langle \sigma\right\rangle }\oplus\ZZ\nicefrac{G}{\left\langle \tau\right\rangle }\oplus\ZZ G\to M$
by setting $\pi(\hat{S})=S,\;\pi(\hat{T})=T$ and $\pi(\hat{A})=A$.
The maps $\pi$ and $\varepsilon$ coincide on $\ZZ\nicefrac{G}{\left\langle \sigma\right\rangle }\oplus\ZZ\nicefrac{G}{\left\langle \tau\right\rangle }$
so we are left to show that $\pi$ is surjective.

Denote by $N_{\sigma},N_{\tau}$ the norm elements $N_{\sigma}={\displaystyle \sum_{i=0}^{n-1}}\sigma^{i}$
and $N_{\tau}={\displaystyle \sum_{j=0}^{m-1}}\tau^{j}$. 
\begin{lem}
The homomorphism $\pi:\ZZ\nicefrac{G}{\left\langle \sigma\right\rangle }\oplus\ZZ\nicefrac{G}{\left\langle \tau\right\rangle }\oplus\ZZ G\to M$
is surjective and $\ker(\pi)$ is spanned by the elements
\begin{eqnarray*}
V_{g} & = & g\left(\sum_{j=0}^{m-1}\tau^{j}\sum_{i=0}^{r^{j}-1}\sigma^{i}(\hat{A})+\frac{r^{m}-1}{n}\cdot\left(\hat{S}\right)+\hat{T}-\sigma(\hat{T})\right)\\
U_{\tau^{j}} & = & \tau^{j}\left(N_{\sigma}(\hat{A})-\hat{S}+r\tau(\hat{S})\right)
\end{eqnarray*}
where $g\in G$ and $0\leq j\leq m-1$.\end{lem}
\begin{proof}
For convenience, we sketch the proof for $S_{3}\cong C_{3}\rtimes C_{2}$.
The general proof is a similar and is left to the reader.

The Cayley graph of $S_{3}$ with respect to $\left\{ \sigma,\tau\right\} $
is

{\footnotesize{}
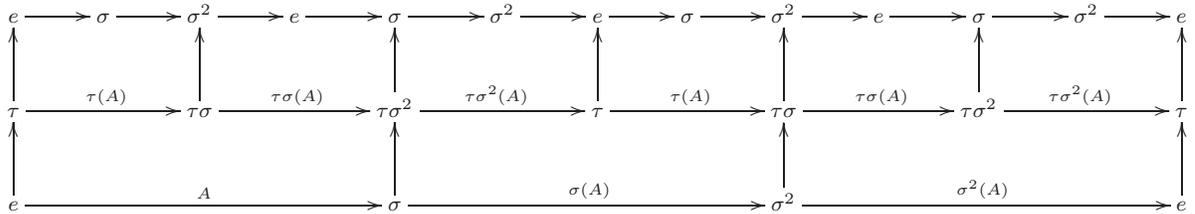
\[
\xymatrix{e\ar[r] & \sigma\ar[r] & \sigma^{2}\ar[r] & e\ar[r] & \sigma\ar[r] & \sigma^{2}\ar[r] & e\ar[r] & \sigma\ar[r] & \sigma^{2}\ar[r] & e\ar[r] & \sigma\ar[r] & \sigma^{2}\ar[r] & e\\
\tau\ar[rr]^{\tau(A)}\ar[u] &  & \tau\sigma\ar[rr]^{\tau\sigma(A)}\ar[u] &  & \tau\sigma^{2}\ar[rr]^{\tau\sigma^{2}(A)}\ar[u] &  & \tau\ar[rr]^{\tau(A)}\ar[u] &  & \tau\sigma\ar[rr]^{\tau\sigma(A)}\ar[u] &  & \tau\sigma^{2}\ar[rr]^{\tau\sigma^{2}(A)}\ar[u] &  & \tau\ar[u]\\
e\ar[rrrr]^{A}\ar[u] &  &  &  & \sigma\ar[rrrr]^{\sigma(A)}\ar[u] &  &  &  & \sigma^{2}\ar[rrrr]^{\sigma^{2}(A)}\ar[u] &  &  &  & e\ar[u]
}
\]
}Note that some of the vertices and edges appear more than one time
to make the illustration clearer. 

By definition, $S,T$ and $A$ are in $Im(\pi)$ , so to show surjectivity
it is enough to show that they generate $M$.

Suppose that $a\in M$ is a flow which contains the edge $\sigma^{2}\to\tau\sigma$
with coefficient $\lambda$. We can remove this edge by moving to
$a-\lambda\sigma(A)$. Similarly we can remove the edge $\sigma\to\tau\sigma^{2}$
with the use of $A$, and without adding back the edge $\sigma^{2}\to\tau\sigma$.
We can do the same trick in the top row and therefore we are left
only with the edges of the form $\sigma^{i}\to\sigma^{i+1}$, $\tau\sigma^{i}\to\tau\sigma^{i+1}$
and $\tau^{j}\to\tau^{j+1}$. Clearly, this flow is a linear combination
of $S,\tau(S)$ (horizontal lines) and $T$ (left vertical line),
concluding that $a\in Im(\pi)$. 

We now turn to study the kernel $\ker(\pi)$. Clearly the norms $N_{\sigma},N_{\tau}$
satisfy $N_{\sigma}\sigma=\sigma N_{\sigma}=N_{\sigma}$ and $N_{\tau}\tau=\tau N_{\tau}=N_{\tau}$.
Since $\left\langle \sigma\right\rangle $ is normal in $G$, the
norm $N_{\sigma}$ is central in $\ZZ G$. 

By definition, the flow $A$ has $+1$ on the right and bottom edges
and $-1$ on the left and top edges. We thus see that $A+\sigma(A)+\sigma^{2}(A)$
is zero on the edges of degree $\tau$ (which point up) and is making
one cycle $e\to\sigma\to\sigma^{2}\to e$ on the bottom lines and
twice the cycle $\tau\to\tau\sigma\to\tau\sigma^{2}\to\tau$ in the
middle with a minus sign. In other words we have that $A+\sigma(A)+\sigma^{2}(A)=S-2\tau(S)$,
so that $U_{\tau^{0}}\in\ker(\pi)$. Similarly, the sum $A+\tau(A)+\tau\sigma(A)$
equals the square with the edges $e\to\sigma$, $\sigma(T)=\sigma\to\tau\sigma^{2}\to\sigma$
with plus sign and the edges $T=e\to\tau\to e$ and $e\to\sigma\to\sigma^{2}\to e\to\sigma$
with a minus sign. The top and the bottom edges gives us one cycle
$e\to\sigma\to\sigma^{2}\to e$ (where $1=\frac{r^{m}-1}{n}$ for
$r=m=2$ and $n=3$). We conclude that $A+\left(\tau(A)+\tau\sigma(A)\right)=\sigma(T)-T-S$
and therefore $V_{e}\in\ker(\pi)$.

Let $N=span\left\{ U_{g},V_{\tau^{j}}\;\mid\; g\in G,\;0\leq j\leq m-1\right\} $.
If $a\in\nicefrac{\ker(\pi)}{N}$, then using the element $U_{\tau^{j}}$
and $V_{g}$ we can find a representative for it of the form 
\[
x=\sum_{j=0}^{m-1}\left(\sum_{i=0}^{n-1}\alpha_{j,i}\tau^{j}\sigma^{i}(\hat{A})\right)+\beta\cdot\hat{T}+\gamma\cdot\hat{S}
\]
where $\alpha_{j},\beta,\gamma\in\ZZ$. For any $0\leq j\leq m-1$
and $1\leq i\leq n-1$, the only elements in $\left\{ \tau^{j_{1}}\sigma^{i_{1}}A\right\} \cup\left\{ S,T\right\} $
that touch the edge $(\tau^{j}\sigma^{i},\tau^{j}\sigma^{i}\tau)$
are $\tau^{j}\sigma^{i}(A)$ (with a minus sign) and $\tau^{j}\sigma^{i-1}A$
(with a plus sign). We conclude that $\alpha_{j,i}=\alpha_{j,i-1}$,
so setting $\alpha_{j}=\alpha_{j,0}$ we see that
\[
x=\sum_{j=0}^{m-1}\alpha_{j}\tau^{j}N_{\sigma}(\hat{A})+\beta\cdot\hat{T}+\gamma\cdot\hat{S}.
\]
Using the elements $U_{\tau^{j}}$ we can find another representative
of the form
\[
x'=\sum_{j=0}^{m-1}\alpha_{j}\tau^{j}\left(\hat{S}-r\tau\hat{S}\right)+\beta\cdot\hat{T}+\gamma\cdot\hat{S}=\beta\hat{T}+\sum_{j=1}^{m-1}\left(\alpha_{j}-r\alpha_{j-1}\right)\tau^{j}\hat{S}+\left(\alpha_{0}-r\alpha_{m-1}+\gamma\right)\hat{S}.
\]
Since $\pi(x')=0$ and the elements $T,\left\{ \tau^{j}S\right\} _{1}^{m}$
are linearly independent, $x'$ must be zero and therefore $\ker(\pi)=N$. 
\end{proof}

We now turn to study the structure of $\ker(\pi)$. Recall that 
\begin{eqnarray*}
V_{g} & = & g\left({\displaystyle \sum_{j=0}^{m-1}}\tau^{j}\sum_{i=0}^{r^{j}-1}\sigma^{i}(\hat{A})+\frac{r^{m}-1}{n}\cdot\left(\hat{S}\right)+\hat{T}-\sigma(\hat{T})\right)\\
U_{\tau^{j}} & = & \tau^{j}\left(N_{\sigma}(\hat{A})-\hat{S}+r\tau(\hat{S})\right).
\end{eqnarray*}
The elements $\left\{ \tau^{j}N_{\sigma}(\hat{A})\right\} _{j=0}^{m-1}$
are $\ZZ$ linearly independent and $U_{e}$ is $\sigma$ invariant,
hence $M_{0}=span_{\ZZ}\left\{ U_{\tau^{j}}\right\} $ is a submodule
of $\ker(\pi)$ isomorphic to $\ZZ\nicefrac{G}{\left\langle \sigma\right\rangle }$.
Obviously $M_{0}\subseteq\ker(\pi)\cap\left\langle \hat{A},\hat{S}\right\rangle _{\ZZ G}$
and we wish to show that this is actually an equality.
\begin{lem}
There is an equality $M_{0}=\ker(\pi)\cap\left\langle \hat{A},\hat{S}\right\rangle _{\ZZ G}$
and $\nicefrac{\ker(\pi)}{M_{0}}\cong I_{G/\left\langle \tau\right\rangle }$.\end{lem}
\begin{proof}
Let $x\in\ker(\pi)\cap\left\langle \hat{A},\hat{S}\right\rangle _{\ZZ G}$.
As in the previous lemma, the coefficients of $\tau^{j}\sigma^{i}(A)$
for $j$ fixed are the same, so $x$ has the form 
\[
{\displaystyle \sum_{0}^{m-1}}\alpha_{j}\tau^{j}N_{\sigma}\hat{A}+{\displaystyle \sum_{j=0}^{m-1}}\beta_{j}\tau^{j}\hat{S}\equiv_{M_{0}}{\displaystyle \sum_{0}^{m-1}}\alpha_{j}\tau^{j}\left(\hat{S}-r\tau(\hat{S})\right)+{\displaystyle \sum_{j=0}^{m-1}}\beta_{j}\tau^{j}\hat{S}=\sum_{0}^{m-1}\left(\alpha_{j}-r\alpha_{j-1}+\beta_{j}\right)\tau^{j}\hat{S}
\]
where the subtraction in the indices is modulo $m$. Since $\pi(x)=0$
and $\left\{ \tau^{j}S\right\} _{1}^{m}$ are linearly independent,
we get that $x\equiv_{M_{0}}0$, so $M_{0}=\ker(\pi)\cap\left\langle \hat{A},\hat{S}\right\rangle _{\ZZ G}$.
Since
\[
\nicefrac{\ker(\pi)}{M_{0}}\leq\nicefrac{\ZZ G\oplus\ZZ\nicefrac{G}{\left\langle \sigma\right\rangle }\oplus\ZZ\nicefrac{G}{\left\langle \tau\right\rangle }}{\left\langle \hat{A},\hat{S}\right\rangle }\cong\ZZ\nicefrac{G}{\left\langle \tau\right\rangle }
\]
and $\nicefrac{\ker(\pi)}{M_{0}}$ is generated by the images $\left\{ g(\hat{T}-\sigma\left(\hat{T}\right))\right\} _{g\in G}$
of $\left\{ V_{g}\right\} _{g\in G}$, it follows that $\nicefrac{\ker(\pi)}{M_{0}}\cong I_{G/\left\langle \tau\right\rangle }$.
\end{proof}
Let $P$ be any permutation $G$-lattice with a surjection $\psi:P\to I_{G/\left\langle \tau\right\rangle }$.
We have the following exact diagram
\[
\xymatrix{ &  & 0\ar[d] & 0\ar[d]\\
 &  & \ker(\psi)\ar[d]\ar[r]^{=} & \ker(\psi)\ar[d]\\
0\ar[r] & \ZZ\nicefrac{G}{\left\langle \sigma\right\rangle }\ar[d]_{=}\ar[r] & Q\ar[r]\ar[d] & P\ar[d]^{\psi}\ar[r] & 0\\
0\ar[r] & \ZZ\nicefrac{G}{\left\langle \sigma\right\rangle }\ar[r] & \ker\left(\pi\right)\ar[r]\ar[d] & I_{G/\left\langle \tau\right\rangle }\ar[r]\ar[d] & 0\\
 &  & 0 & 0
}
\]
where $Q=\ker(\pi)\times_{I_{G/\left\langle \tau\right\rangle }}P$.
Since $P$ and $\ZZ\nicefrac{G}{\left\langle \sigma\right\rangle }$
are permutation lattices, the middle row splits by \lemref{splitting_invertibles}
and hence $Q$ is a permutation lattice as well. The lattice $\ker(\pi)$
is invertible so the middle column is a flasque resolution, hence
$\left[M\right]^{fl}=\left[\ker(\pi)\right]=\left[\ker(\psi)\right]^{fl}$.
On the other hand, the exact sequence on the right column corresponds
to flows on the graph with edges $V=\nicefrac{G}{\left\langle \tau\right\rangle }$.
Using \thmref{reduction_by_flasque} we conclude the following.
\begin{thm}
\label{thm:faithful_graphs}The field $\FF(\ZZ G\oplus Fl(G))^{G}$
is stably isomorphic to $\FF(\ZZ G\oplus Fl(\nicefrac{G}{\left\langle \tau\right\rangle }))^{G}$
over $\FF$.
\end{thm}
While the action of $G$ on $Cay(G,G)$ is faithful, its action on
a graph with vertices $\nicefrac{G}{\left\langle \tau\right\rangle }$
has kernel at least $Z(G)\cap\left\langle \tau\right\rangle $. As
we shall see in the next section, this kernel can be mod out, namely
that $\FF(\ZZ G\oplus Fl(\nicefrac{G}{\left\langle \tau\right\rangle }))^{G}$
is stably isomorphic to $\FF(\ZZ\nicefrac{G}{\left\langle \tau\right\rangle }\oplus Fl(\nicefrac{G}{\left\langle \tau\right\rangle }))^{G/\left\langle \tau\right\rangle }$,
thus proving part (2) of \thmref{Main_Theorem_sylow}.
\begin{rem}
Let $G=\left\langle \sigma,\tau\;\mid\;\sigma^{n}=\tau^{2m}=e,\;\tau\sigma\tau^{-1}=\sigma^{-1}\right\rangle $
with $(n,2m)=1$ so that $\nicefrac{G}{Z(G)\cap\left\langle \tau\right\rangle }\cong D_{2n}$
is a dihedral group where $n$ is odd. Recall that Snider proved that
if $\FF$ contains a primitive $2n$ root of unity, then $\FF(D_{2n})^{D_{2n}}/\FF$
is stably rational, hence $\FF(G)^{G}/\FF$ is also stably rational.
In \cite{endo_classification_1975}, Endo and Miyata also proved that
for such groups the lattice $I_{G}^{*}$ is a quasi-permutation lattice,
which is equivalent to $Fl(G,S)$ being stably permutation. In particular,
this means that $\FF(G)^{G}=\FF(Fl(G,S)\oplus\ZZ G)^{G}$ is stably
isomorphic to $\FF(\ZZ G)^{G}$, so that $\FF(\ZZ G)^{G}/\FF$ is
a stably rational extension. 
\end{rem}

\subsection{\label{sub:Semidirect-product}Semidirect product of abelian groups}

Let $G=N\rtimes H$ be a semidirect product of abelian groups. In
this case we consider the set of edges $V=\nicefrac{G}{H}$, with
some fixed choice of copy of $H$ in $G$. The $G$-action on $V$
is not necessarily faithful. More precisely, if $h\in G$ is in the
kernel, then it must be in $H$, and for every $g\in N$ we also have
\[
gH=h\left(gH\right)=hgh^{-1}H
\]
where $hgh^{-1}\in N$, so that $g=hgh^{-1}$. It follows that the
kernel is exactly $H_{0}:=Z(G)\cap H$. Setting $\tilde{G}=\nicefrac{G}{H_{0}}$,
$\tilde{H}=\nicefrac{H}{H_{0}}$, we see that the $G$-action on $V$
factors through $\tilde{G}$, and as such the structure of $V$ is
$\tilde{V}=\nicefrac{\tilde{G}}{\tilde{H}}$.

Let $E_{-}$ be the set of all directed edges on $V\cong\tilde{V}$
which are not loops. The graphs $X=(\nicefrac{G}{H},E_{-}(G))$ and
$\tilde{X}=(\nicefrac{\tilde{G}}{\tilde{H}},E_{-}(\tilde{G}))$ are
isomorphic and the action of $G$ on $X\cong\tilde{X}$ is factored
through $\tilde{G}$, and it follows that $\FF(Fl(X))^{G}=\left(\FF(Fl(X))^{H_{0}}\right)^{\tilde{G}}=\FF(Fl(\tilde{X}))^{\tilde{G}}$
under the identification of $X$ and $\tilde{X}$. 

Recall that our usual setting is $\FF(Fl(X)\oplus\ZZ G)^{G}$ and
we ask whether it is rational over $\FF(Fl(X))^{G}$. Writing $\KK=\FF(Fl(X))$,
so that $G$ act on $\KK$ with kernel $H_{0}$, the question is the
rationality of $\KK(\ZZ G)^{G}/\KK^{G}$. Two standard reasons used
to show the rationality of such extensions are \propref{Masuda},
where we need $G$ to act faithfully on $\KK$, and Fischer's theorem
where we need $G$ to be abelian of some exponent $m$, $\KK^{G}=\KK$
and that $\KK$ contains a primitive root of unity of order $m$.
While both of these conditions aren't true in this case, the next
theorem shows how to combine them.
\begin{lem}
\label{lem:Fischer-generalize}Let $G=N\rtimes H$ be a semidirect
product with $H$ abelian and set $H_{0}=Z(G)\cap H$ - the kernel
of the $H$-action on $N$. Let $G$ act on a field $\KK$ with kernel
$H_{0}$, and assume that $\KK^{G}$ contains a primitive root of
unity of order $\left|H\right|$. Then $\KK(\ZZ\nicefrac{G}{N})^{G}$
is rational over $\KK^{G}$, and in particular $\KK(\ZZ G)^{G}$ is
stably rational over $\KK^{G}$.\end{lem}
\begin{proof}
Note first that $G$ acts faithfully on $\KK(\ZZ\nicefrac{G}{N})$
so
\[
\KK(\ZZ\nicefrac{G}{N})^{G}\approx\KK(\ZZ\nicefrac{G}{N}\oplus\ZZ G)^{G}\approx\KK(\ZZ G)^{G},
\]
hence it is enough to prove the rationality of $\KK(\ZZ\nicefrac{G}{N})^{G}/\KK^{G}$.

Since $N$ acts trivially on $\ZZ\nicefrac{G}{N}$ it follows that
$\KK(\ZZ\nicefrac{G}{N})^{N}=\KK^{N}\left(\ZZ\nicefrac{G}{N}\right)$
as $\nicefrac{G}{N}\cong H$ field. We are thus reduced to the question
of rationality of $\KK(\ZZ H)^{H}$ over $\KK^{H}$ where the kernel
of the $H$ action on $\KK$ is exactly $H_{0}$ and $\KK^{H}$ contains
an $m=\left|H\right|$ primitive root of unity $\zeta$.

Write $\KK(\ZZ H)=\KK(x_{h}\;\mid\; h\in H)$. Since $\KK^{H}$ contains
a root of unity of the order of the exponent of $H$, we get that
$H^{*}:=Hom(H,\KK^{\times})=Hom(H,(\KK^{H})^{\times})$. For any $\varphi\in H^{*}$
define $x_{\varphi}=\sum_{h}\varphi(h^{-1})x_{h}$. We have that $\KK(x_{h}\;\mid\; h\in H)=\KK(x_{\varphi}\;\mid\;\varphi\in H^{*})$
and $g(x_{\varphi})=\varphi(g)x_{\varphi}$. Let $\psi:\prod_{\left|H\right|}\ZZ\to H_{0}^{*}$
be defined by 
\[
\psi(k_{1},...,k_{m})(h)=\prod\varphi_{i}^{k_{i}}(h).
\]
Note that since $H$ is abelian and $\KK$ contains a primitive $m$-th
root of unity, the natural homomorphism $H^{*}\to H_{0}^{*}$ is surjective,
and in particular $\psi$ is surjective. The kernel of $\psi$ is
a subgroup of $\prod_{\left|H\right|}\ZZ$ of finite index $\left|H_{0}\right|$,
so it is also free of rank $m$. If $\bar{k}=(k_{1},...,k_{m})$ and
$x^{(\bar{k})}=\prod x_{\varphi_{i}}^{k_{i}}$, then $g(x^{\left(\bar{k}\right)})=\psi(\bar{k})(g)\cdot x^{\left(\bar{k}\right)}$.

Let $\bar{k}^{(j)}=(k_{1}^{(j)},...,k_{m}^{(j)});\; j=1,...,m$ be
a basis for $\ker(\psi)$ and set $x^{(j)}:=\prod x_{\varphi_{i}}^{k_{i}^{(j)}}$
for $1\leq j\leq m$. By the previous argument we have that $\KK(x^{(1)},...,x^{(m)})\subseteq\KK(x_{\varphi_{1}},...,x_{\varphi_{m}})^{H_{0}}$.
If $\varphi_{i}\mid_{K}$ has order $t$ in $K^{*}$, then $x_{\varphi_{i}}^{t}$
is an invariant monomial, and therefore 
\[
\left[\KK(x^{(1)},...,x^{(m)})\left[x_{\varphi_{i}}\right]:\KK(x^{(1)},...,x^{(m)})\right]\leq t.
\]
Using the decomposition of $H_{0}^{*}$ into cyclic groups and adding
each of their generators, we get that
\[
\left|H_{0}\right|=\left[\KK(x_{\varphi_{1}},...,x_{\varphi_{m}}):\KK(x_{\varphi_{1}},...,x_{\varphi_{m}})^{K}\right]\leq\left[\KK(x_{\varphi_{1}},...,x_{\varphi_{m}}):\KK(x^{(1)},...,x^{(m)})\right]\leq\left|H_{0}\right|,
\]
and the equality $\KK(x_{\varphi_{1}},...,x_{\varphi_{m}})^{H_{0}}=\KK(x^{(1)},...,x^{(m)})$
follows (up to this point we followed Fischer's proof).

For each $1\leq j\leq m$ and $h\in\nicefrac{H}{H_{0}}$ we have $h(x^{(j)})=\zeta_{h,j}x^{(j)}$
for elements (roots of unity) $\zeta_{h,j}\in\KK$. For any fixed
$j$, the map $h\mapsto\zeta_{h,j}$ is a 1-cocycle in $Z^{1}\left(\nicefrac{H}{H_{0}},\KK^{\times}\right)$,
which by Hilbert 90 has the form $\zeta_{h,j}=\frac{a_{j}}{h(a_{j})}$
for some $a_{j}\in\KK^{\times}$. It follows that $y^{(j)}=x^{(j)}a_{j}$
are $\nicefrac{H}{H_{0}}$ invariant and algebraically independent.
The proof is now finished by noting that 
\[
\KK(x^{(1)},...,x^{(m)})^{\left(H/H_{0}\right)}=\KK(y^{(1)},...,y^{(m)})^{\left(H/H_{0}\right)}=\KK^{\left(H/H_{0}\right)}(y^{(1)},...,y^{(m)})=\KK^{H}(y^{(1)},...,y^{(m)}).
\]
\end{proof}
\begin{thm}
\label{thm:Now_we_are_faithful}Let $G=N\rtimes H$ be a semidirect
product of abelian groups and set $\tilde{G}=\nicefrac{G}{H\cap Z(G)}$.
Then setting $X=(\nicefrac{G}{H},E_{-}(G))$ and $\tilde{X}=(\nicefrac{\tilde{G}}{\tilde{H}},E_{-}(\tilde{G}))$
we have $\FF(Fl(X)\oplus\ZZ G)^{G}\approx\FF(Fl(\tilde{X})\oplus\ZZ\tilde{G})^{\tilde{G}}$.
\end{thm}
Thus, by the theorem above we assume for the rest of this section
that $H$ acts on $N$ faithfully. In particular, if the new group
$G$ is dihedral $D_{2n}$ with $n$ odd and $\FF$ contains a primitive
$2n$ root of unity, then \thmref{faithful_graphs} together with
Snider's result shows that $\FF(Fl(\nicefrac{G}{\left\langle \tau\right\rangle })\oplus\ZZ G)^{G}/\FF$
is a stably rational extension.

\subsection{$G=H\times K$ for groups $H,K$ of coprime orders }

The aim of this section is to show that the field generated by $\FF(Fl(H)\oplus\ZZ H)^{H}\otimes\FF(Fl(K)\oplus\ZZ K)^{K}$
is stably isomorphic to $\FF(Fl(G)\oplus\ZZ G)^{G}$ whenever $G=H\times K$
such that $H$ and $K$ have coprime orders. In particular this means
that in order to establish stable rationality for the group $G$,
it is enough to prove it for $H$ and $K$.

Let $G$ be any finite group and consider its action on the graph
$Cay(G,G_{-})=(V,E)$ and its corresponding lattices $\ZZ V,\ZZ E$
and $Fl(V,E)$. A natural $\ZZ$-basis for $Fl(V,E)$ is given by
\begin{eqnarray*}
d(g,h) & = & (e\to g\to gh)-(e\to gh),\; g,h,gh\in G\backslash\left\{ e\right\} ,\\
d(g,g^{-1}) & = & \left(e\to g\right)-\left(g\to e\right),\; g\in G\backslash\left\{ e\right\} .
\end{eqnarray*}
In other words, we take the spanning tree consisting of all the edges
$e\to g$ and for each edge not appearing in this tree we create a
simple flow by adding to the edge a simple path from the tree. It
is now easy to show that these flows satisfy the two cocycle condition,
namely
\[
d(g_{1},g_{2})+d(g_{1}g_{2},g_{3})=d(g_{1},g_{2}g_{3})+g_{1}(d(g_{2},g_{3})),
\]
where we write $d(e,g)=d(g,e)=1$ for each $g\in G$. Furthermore,
these relations generate all the other relations on this basis. This
is not surprising since this is just the next part of the bar resolution
of $G$ which starts with
\[
\xymatrix{\ZZ E\ar[r]^{\partial} & \ZZ V\ar[r]^{\varepsilon} & \ZZ\ar[r] & 0}
.
\]
Letting $d_{g}$ be the edge $e\to g$ for $g\neq e$ we get that
flows $d(g,h)$ together with the edges $d_{g}$ constitute a basis
for $\ZZ E$, and additionally we have
\[
d(h,g)=d_{h}+h(d_{g})-d_{hg}.
\]
Note that applying $\partial$ on the equation above we get that $h(\tilde{d}_{g})=\tilde{d}_{hg}-\tilde{d}_{h}$
which are the defining relation for the lattice $I_{G}$ (where we
identify $\tilde{d}_{g}$ with $g-e$). Keeping all of this in mind,
we have the following definitions (in multiplicative form).

Let $\mathbb{L}/\mathbb{L}^{G}$ be a $G$-Galois extension of fields
and $\alpha\in Z^{2}(G,\LL^{\times})$ a $2$-cocycle. Define $\LL_{\alpha}(I_{G})$
to be the field $\LL\left(y_{g}\;\mid\; e\neq g\in G\right)$ where
the $y_{g}$ are algebraically independent with the $G$-Galois action
\[
h(y_{g})=\frac{y_{hg}}{y_{h}}\alpha(h,g),
\]
where we denote $y_{e}=1$. Note that the $G$-action can be rewritten
as $\frac{y_{h}h(y_{g})}{y_{hg}}=\alpha(h,g)$. 

For fixed elements $\lambda_{g}\in\LL^{\times},\; e\neq g\in G$,
let $z_{g}=\lambda_{g}y_{g}$. Then clearly $\LL_{\alpha}(I_{G})=\LL(z_{g}\;\mid\; e\neq g\in G)$
with the $G$-action defined by $\frac{z_{h}h(z_{g})}{z_{hg}}=\frac{\lambda_{h}h(\lambda_{g})}{\lambda_{hg}}\alpha(h,g)$,
so in particular the field $\LL_{\alpha}(I_{G})$ is a function of
the cohomology class of $\alpha$.

It can also be shown that $\LL_{\alpha}(I_{G})$ is a generic splitting
field of $\Delta=\Delta(\LL/\LL^{G},G,\alpha)$ in the sense of Amitsur
\cite{amitsur_generic_1982}. We need only one property of such fields
which we now prove.
\begin{lem}
\label{lem:splitting_field_is_rational-1}Let $\Delta=\Delta(\LL/\LL^{G},G,\alpha)$
where $\Delta$ is split. Then $\LL_{\alpha}(I_{G})^{G}$ is stably
rational over $\LL^{G}$.\end{lem}
\begin{proof}
Since $\Delta$ splits, the cocycle $\alpha$ is cohomologous to $1$.
As mentioned above, we may assume that the action is defined by $h(y_{g})=\frac{y_{hg}}{y_{h}}$,
or in other words, $\LL_{\alpha}(I_{G})\cong\LL_{1}(I_{G})\cong\LL(I_{G})$
where the correspondence is $g-e\leftrightarrow y_{g}$. The module
$I_{G}$ is part of the exact sequence
\[
\xymatrix{0\ar[r] & I_{G}\ar[r] & \ZZ G\ar[r]^{\varepsilon} & \ZZ\ar[r] & 0}
,
\]
so by \thmref{reduction_by_permutation} we get that $\LL(\ZZ G)^{G}$
is rational over $\LL(I_{G})^{G}$. On the other hand $\LL(\ZZ G)^{G}$
is rational over $\LL^{G}$ since $\ZZ G$ is a permutation lattice
and $\LL/\LL^{G}$ is Galois, and the lemma follows.
\end{proof}
We now turn to prove that if $\FF(Fl(H))^{H}$ and $\FF(Fl(K))^{K}$
are stably rational over $\FF$ and $H,K$ are of coprime orders,
then so is $\FF(Fl(G))^{G}$ for $G=H\times K$. 

This was proved in the nongraded case by Katsylo \cite{katsylo_stable_1990},
Schofield \cite{schofield_matrix_1992} and Saltman \cite{saltman_note_1992}.
We will adapt Saltman's proof for the graded case. 

Let $\EE$ be a field with a $G$-action and consider the field $\EE\left(Fl(V,E)\right)$.
Letting $c(g,h)$ be the variables in $\EE(Fl(V,E))$ corresponding
$d(g,h)$, it follows that $\left\{ c(g,h)\;\mid\; e\neq g,h\in G\right\} $
are algebraically independent over $\EE$ and $\sigma(c(g,h))=\frac{c(\sigma,g)c(\sigma g,h)}{c(\sigma,gh)}$
for every $\sigma,g,h\in G$, where we denote $c(g,e)=c(e,g)=1$ for
all $g\in G$. We call such a field a generic $G$ 2-cocycle extension
of $\EE$, and denote it by $\EE(c)$. 
\begin{lem}
\label{lem:generic_2_cocycle-1}Let $\LL/\LL^{G}$ be a $G$-Galois
extension $\alpha\in Z^{2}(G,\LL^{\times})$ and let $c$ be a generic
$G$ 2-cocycle. Then we have the following:
\begin{enumerate}
\item $\LL(c)_{\alpha\cdot c}(I_{G})\cong\LL(c)_{c}(I_{G})$.
\item $\LL(c)_{c}(I_{G})^{G}$ is a rational extension of $\LL^{G}$.
\end{enumerate}
\end{lem}
\begin{proof}

\begin{enumerate}
\item The action on $\LL_{\alpha c}(I_{G})$ is defined by
\[
\frac{y_{h}h(y_{g})}{y_{hg}}=\alpha(h,g)c(h,g).
\]
Setting $c'(h,g)=c(h,g)\alpha(h,g)$, we clearly get that $\left\{ c'(h,g)\;\mid\; e\neq h,g\in G\right\} $
are algebraically independent over $\LL$, so $c'$ is a generic $G$
2-cocycle and $\LL(c)_{\alpha\cdot c}(I_{G})\cong\LL(c')_{c'}(I_{G})$.
\item As mentioned in the beginning of this section, $\LL(c)_{c}(I_{G})$
is just a complex way of writing $\LL(\ZZ E)$. Since $\LL/\LL^{G}$
is $G$-Galois and $\ZZ E$ is a permutation lattice, we get that
$\LL(c)_{c}(I_{G})^{G}=\LL(\ZZ E)^{G}$ is rational over $\LL^{G}$.
\end{enumerate}
\end{proof}
We now turn to decompose 2-cocycles on $G$ into two 2-cocycles on
$H$ and $K$.

Let $c$ be a generic $G$ 2-cocycle and let $\alpha,\beta$ be inflations
of generic $H\cong\nicefrac{G}{K}$ and $K\cong\nicefrac{G}{H}$ 2-cocycles.
More precisely, the elements $\left\{ \alpha(h_{1},h_{2})\;\mid\; e\neq h_{1},h_{2}\in H\right\} $
are algebraically independent, $\alpha(h_{1}g_{1},h_{2}g_{2})=\alpha(h_{1},h_{2})$
for any $h_{1},h_{2}\in H$ and $g_{1},g_{2}\in K$, and the $G$-action
is defined by
\begin{eqnarray*}
g(\alpha(h_{1},h_{2})) & = & \alpha(h_{1},h_{2})\quad\mbox{for all }g\in K,\; h_{1},h_{2}\in H\\
h_{1}(\alpha(h_{2},h_{3})) & = & \frac{\alpha(h_{1},h_{2})\alpha(h_{1}h_{2},h_{3})}{\alpha(h_{1},h_{2}h_{3})}\mbox{ for all }h_{1},h_{2},h_{3}\in H,
\end{eqnarray*}
and similarly define $\beta$. 

We will show that $\FF(\alpha,\beta,c)_{\alpha\beta c}(I_{G})^{G}$
is rational over $\FF(c)^{G}$ and over $\FF(\alpha,\beta)^{G}$.
The last field is the fraction field of $\FF(\alpha)^{H}\otimes_{\FF}\FF(\beta)^{K}$,
and this will finish the proof.

The second part follows from the previous two lemmas. More precisely
we have $\FF(\alpha,\beta,c)_{\alpha\beta c}(I_{G})^{G}=\FF(\alpha,\beta)(c)_{\alpha\beta c}(I_{G})^{G}\cong\FF(\alpha,\beta)(c)_{c}(I_{G})^{G}$
which is rational over $\FF(\alpha,\beta)^{G}$.

Since $\alpha,\beta$ are two generic $2$-cocycles for $H,K$ respectively,
we expect that $\alpha\cdot\beta$ will have similar properties as
a generic $G$ 2-cocycle. Indeed, the next lemma shows that the previous
lemma still holds in this case.
\begin{lem}
\label{lem:product_generic_2_cocycle-1}Let $\LL/\LL^{G}$ be a $G$-Galois
extension where $G=H\times K$ with $\left(\left|H\right|,\left|K\right|\right)=1$.
Let $\alpha,\beta$ be the inflations of the $H$ and $K$ generic
2-cocycles. Then the following holds:
\begin{enumerate}
\item $\LL(\alpha,\beta)_{\alpha\beta\gamma}(I_{G})\cong\LL(\alpha,\beta)_{\alpha\beta}(I_{G})$
for all $\gamma\in Z^{2}(G,\LL^{\times})$.
\item $\LL(\alpha,\beta)_{\alpha\beta}(I_{G})^{G}$ is a stably rational
extension of $\LL^{G}$. 
\end{enumerate}
\end{lem}
\begin{proof}

\begin{enumerate}
\item Since $H^{1}(K,\LL^{\times})=0$ by Hilbert 90, we have the inflation
restriction exact sequence
\[
\xymatrix{0\ar[r] & H^{2}(\nicefrac{G}{K},\left(\LL^{\times}\right)^{K})\ar[r]^{inf} & H^{2}(G,\LL^{\times})\ar[r]^{res} & H^{2}(K,\LL^{\times})^{G/K}}
.
\]
Since $\left|H\right|=\left|\nicefrac{G}{K}\right|$ and $\left|K\right|$
are coprime, we can find $a,b\in\ZZ$ such that $a\left|\nicefrac{G}{K}\right|+b\left|K\right|=1$.
If $\left[\gamma\right]\in H^{2}(G,\LL^{\times})$ is any 2-cocycle,
then we can write $\gamma=\gamma_{H}\cdot\gamma_{K}$ where $\gamma_{H}=\gamma^{b\left|K\right|}$
and $\gamma_{K}=\gamma^{a\left|H\right|}$. The group $H^{2}(K,\LL^{\times})$
is $\left|K\right|$-torsion, so $res(\gamma_{H})=0$, and therefore
$\gamma_{H}\in inf(Z^{2}(\nicefrac{G}{K},\left(\LL^{\times}\right)^{K})$.
Switching the roles of $H$ and $K$ we get that $\gamma_{K}\in inf(Z^{2}(\nicefrac{G}{H},\left(\LL^{\times}\right)^{H})$.\\
Letting $y_{g}$ be the indeterminates corresponding to $I_{G}$ in
$\LL(\alpha,\beta)_{\alpha\beta\gamma}(I_{G})$, we have the action
\begin{eqnarray*}
h_{1}\tau_{1}(y_{h_{2}\tau_{2}}) & = & \frac{y_{h_{1}h_{2}\tau_{1}\tau_{2}}}{y_{h_{1}\tau_{1}}}\alpha(h_{1}\tau_{1},h_{2}\tau_{2})\beta(h_{1}\tau_{1},h_{2}\tau_{2})\gamma(h_{1}\tau_{1},h_{2}\tau_{2})\\
 & = & \frac{y_{h_{1}h_{2}\tau_{1}\tau_{2}}}{y_{h_{1}\tau_{1}}}\alpha(h_{1},h_{2})\beta(\tau_{1},\tau_{2})\gamma_{H}(h_{1},h_{2})\gamma_{K}(\tau_{1},\tau_{2})
\end{eqnarray*}
for all $h_{1},h_{2}\in H$ and $\tau_{1},\tau_{2}\in K$. As in the
previous lemma, defining $\tilde{\alpha}(g_{1},g_{2})=\alpha(g_{1},g_{2})\gamma_{H}(g_{1},g_{2})$
and $\tilde{\beta}(g_{1},g_{2})=\beta(g_{1},g_{2})\gamma_{K}(g_{1},g_{2})$
for $g_{1},g_{2}\in G$, we get that $\tilde{\alpha}$ and $\tilde{\beta}$
are inflation of generic 2-cocycles in $H,K$ respectively, which
show that $ $$\LL(\alpha,\beta)_{\alpha\beta\gamma}(I_{G})\cong\LL(\alpha,\beta)_{\alpha\beta}(I_{G})$.
\item Consider the field $\KK=\LL(\alpha,\beta)_{\alpha}(I_{G/K})_{\beta}(I_{G/H})_{\alpha\beta}(I_{G})$,
where the indeterminates of $I_{G/K},I_{G/H}$ and $I_{G}$ are denoted
by $z_{h},w_{\tau}$ and $y_{g}$ respectively.\\
The 2-cocycle $\alpha\cdot\beta$ splits in $\LL(\alpha,\beta)_{\alpha}(I_{G/K})_{\beta}(I_{G/H})$.
Indeed, setting $\lambda_{h\tau}=z_{h}w_{\tau}$ for $h\in H$ and
$\tau\in K$ we get that $\alpha\beta(g_{1},g_{2})=\frac{\lambda_{g_{1}}g_{1}(\lambda_{g_{2}})}{\lambda_{g_{1}g_{2}}}$,
where we use the fact that the $z_{h}$ are $K$-trivial and the $w_{\tau}$
are $H$-trivial. We conclude that $\KK\cong\LL(\alpha,\beta)_{\alpha}(I_{G/K})_{\beta}(I_{G/H})_{1}(I_{G})$,
so $\KK^{G}$ is rational over $\LL(\alpha,\beta)_{\alpha}(I_{G/K})_{\beta}(I_{G/H})^{G}$.
Taking the invariants, first under $K$ and then under $H$, we see
that the last field is rational over $\LL^{G}$.\\
On the other hand, we have $\KK=\LL(\alpha,\beta)_{\alpha\beta}(I_{G})_{\beta}(I_{G/H})_{\alpha}(I_{G/K})$,
and we wish to show that $\alpha$ splits in $\LL(\alpha,\beta)_{\alpha\beta}(I_{G})^{K}$
so we can drop the last term $I_{G/K}$. Let $a,b\in\ZZ$ such that
$\left|H\right|a+\left|K\right|b=1$, and set $\tilde{z}_{h}=z_{h}\cdot\left({\displaystyle \prod_{\tau\in K}}\tau(y_{h})\right)^{-b}$
for each $h\in H$ and note that $\tilde{z}_{h}$ is $K$ invariant.
For $h,h'\in H$ we get that 
\begin{eqnarray*}
h'\left(\tilde{z}_{h}\right) & = & \left[\frac{z_{h'h}}{z_{h'}}\alpha(h',h)\right]/\left[\prod_{\tau}\tau\left(\frac{y_{h'h}}{y_{h'}}\alpha(h',h)\beta(h',h)\right)^{b}\right]\\
 & = & \left[\frac{z_{h'h}}{z_{h'}}\alpha(h',h)\right]/\left[\prod_{\tau}\tau\left(\frac{y_{h'h}}{y_{h'}}\right)^{a}\left(\alpha(h',h)\right)^{b\left|K\right|}\right]=\frac{\tilde{z}_{h'h}}{\tilde{z}_{h'}}\frac{\alpha(h',h)}{\alpha(h',h)^{b\left|K\right|}}\\
 & = & \frac{\tilde{z}_{h'h}}{\tilde{z}_{h'}}\alpha(h',h)^{a\left|H\right|}.
\end{eqnarray*}
Since $\alpha$ is an inflation of an $H$ 2-cocycle, we get that
$\alpha^{a\left|H\right|}\sim1$ where both cocycles are in $\LL(\alpha,\beta)^{K}$.
Thus, we can find $\lambda_{h}\in\LL(\alpha,\beta)^{K}$ such that
$\alpha(h_{1},h_{2})=\frac{\lambda_{h_{1}}h_{1}(\lambda_{h_{2}})}{\lambda_{h_{1}h_{2}}}$
for all $h_{1},h_{2}\in H$. Taking $\hat{z}_{h}=\tilde{z}_{h}\cdot\lambda_{h}$
we get that $(1)$ $\hat{z}_{h}$ are $K$ invariant and $(2)$ $H$
acts on them by $h'(\hat{z}_{h})=\frac{\hat{z}_{h'h}}{\hat{z}_{h'}}$,
or in other words $\KK\cong\LL(\alpha,\beta)_{\alpha\beta}(I_{G})_{\beta}(I_{G/H})_{1}(I_{G/K})$.
By \lemref{splitting_field_is_rational-1}, we know that $\KK^{H}$
is stably isomorphic to $\KK=\LL(\alpha,\beta)_{\alpha\beta}(I_{G})_{\beta}(I_{G/H})^{H}$,
and actually it can be shown that this stably isomorphism uses $K$
invariant indeterminates, so we actually have that $\KK^{G}$ is stably
isomorphic to $\LL(\alpha,\beta)_{\alpha\beta}(I_{G})_{\beta}(I_{G/H})^{G}$.\\
Repeating this process for $K$, we get that $\KK^{G}$ is stably
isomorphic to $\LL(\alpha,\beta)_{\alpha\beta}(I_{G})^{G}$ and we
are done.
\end{enumerate}
\end{proof}
We now have all the ingredients for the proof of \thmref{coprime_invariants}.
\begin{proof}
(of \thmref{coprime_invariants}) Let $c$ be a generic $G$ 2-cocycle
where $G=H\times K$ with $\left(\left|H\right|,\left|K\right|\right)=1$.
Let $\alpha,\beta$ be the inflations of the $H$ and $K$ generic
2-cocycles. Consider the field $\KK=\FF(\alpha,\beta,c)_{\alpha\beta c}(I_{G})^{G}$.
From part (1) in \lemref{generic_2_cocycle-1} we get that 
\[
\FF(\alpha,\beta,c)_{\alpha\beta c}(I_{G})=\left(\FF(\alpha,\beta)\right)(c)_{\alpha\beta c}(I_{G})\cong\left(\FF(\alpha,\beta)\right)(c)_{c}(I_{G})
\]
and from part (2) we get that $\left(\FF(\alpha,\beta)\right)(c)_{c}(I_{G})^{G}$
is rational over $\FF(\alpha,\beta)^{G}$. \lemref{product_generic_2_cocycle-1}
shows similarly that $\FF(\alpha,\beta,c)_{\alpha\beta c}(I_{G})^{G}$
is stably isomorphic to $\FF(c)^{G}$.

Finally, we have
\begin{eqnarray*}
\left|G\right|=\left[\FF(\alpha,\beta):\FF(\alpha,\beta)^{G}\right] & \leq & \left[\FF(\alpha,\beta):\FF(\alpha)^{H}\FF(\beta)^{K}\right]=\left[\FF(\alpha,\beta):\FF(\alpha)\FF(\beta)^{K}\right]\left[\FF(\alpha)\FF(\beta)^{K}:\FF(\alpha)^{H}\FF(\beta)^{K}\right]\\
 & \leq & \left[\FF(\beta):\FF(\beta)^{K}\right]\left[\FF(\alpha):\FF(\alpha)^{H}\right]=\left|K\right|\left|H\right|=\left|G\right|,
\end{eqnarray*}
so there are equalities everywhere. In particular $\FF(\alpha,\beta)^{G}=\FF(\alpha)^{H}\cdot\FF(\beta)^{K}$,
which is isomorphic to the fraction field of $\FF(\alpha)^{H}\otimes_{\FF}\FF(\beta)^{K}$
(since $\FF(\alpha)\cap\FF(\beta)=\FF$). It follows that $\FF(c)^{G}$
is stably isomorphic to the fraction field of $\FF(\alpha)^{H}\otimes_{\FF}\FF(\beta)^{K}$.
\end{proof}

\section{\label{sec:The-ungraded-center}The center of the generic division
algebra}

The center of the generic division algebra of degree $n$ can also
be represented using flows in graphs. More precisely, let $V$ be
a set of vertices of cardinality $n$ and let $E_{-}=\left\{ \left(v,u\right)\;\mid\; v,u\in V\;\mbox{distinct}\right\} $.
The graph $X=(V,E_{-})$ has a natural $S_{n}$-action, and the center
of the generic division algebra of degree $n$ is $\FF(Fl(X)\oplus\ZZ\nicefrac{S_{n}}{S_{n-1}})^{S_{n}}$
(up to stable isomorphism). Actually the part $\ZZ\nicefrac{S_{n}}{S_{n-1}}$
can be thought of as adding the self loops to $E_{-}$.

We finish this paper by recalling some of the results regarding the
extension $\FF(Fl(X)\oplus\ZZ\nicefrac{S_{n}}{S_{n-1}})^{S_{n}}/\FF$
and reinterpret them using the language of flows.\\

For $n=2$ the group $S_{n}=C_{2}$ is just the cyclic group of order
$2$, so that $Fl(X)\cong\ZZ$ because $X$ is just a cycle of order
2 (see also \corref{flows_for_cyclic_groups}). The required field
is thus $\FF(\ZZ\oplus\ZZ C_{2})^{C_{2}}\cong\FF(\ZZ C_{2})^{C_{2}}(x)$
which is rational over $\FF$ since $\FF(\ZZ C_{2})^{C_{2}}/\FF$
is rational.

For $n=3$ the group $S_{3}=D_{6}=\left\langle \sigma,\tau\;\mid\;\sigma^{3}=\tau^{2}=1,\;\tau\sigma\tau^{-1}=\sigma^{-1}\right\rangle $
and $V=\nicefrac{G}{\left\langle \tau\right\rangle }$ as a $G$-set.
Since all the Sylow subgroups of $S_{3}$ are cyclic, \thmref{Main_Theorem_sylow}
and \corref{generalized_diehedral} show that $\FF(X)^{S_{3}}/\FF$
is stably rational (though we assume that $\FF$ contains a primitive
root of unity of order 3).\\

In case $n=p$ is prime, Saltman showed that the extension $\FF(Fl(X))^{S_{n}}/\FF$
is retract rational \cite{saltman_retract_1984}. Later on it was
proved that $Fl(X)$ is actually invertible (see Theorem 9.12 in \cite{colliot-thelene_principal_1987}
and section 3.1 in \cite{bessenrodt_stable_1991}). The retract rationality
now follows since $\FF(Fl(X)\oplus\ZZ\nicefrac{S_{n}}{S_{n-1}})^{S_{n}}/\FF(\ZZ\nicefrac{S_{n}}{S_{n-1}})^{S_{n}}$
is a retract rational extension (\cite{saltman_retract_1984}, Theorem
3.14) and $\FF(\ZZ\nicefrac{S_{n}}{S_{n-1}})^{S_{n}}/\FF$ is rational
by the fundamental theorem of symmetric polynomials.

To prove that $Fl(X)$ is invertible, recall that a $G$-lattice $M$
is invertible if and only if $M_{H_{i}}$, its restriction to the
$H_{i}$-action, is invertible as an $H_{i}$-lattice for subgroups
$H_{i}$ where $gcd\left(\left[G:H_{1}\right],...,\left[G:H_{k}\right]\right)=1$
(see the argument leading to \corref{invertible_by_subgroup}). For
$G=S_{n}$ choose $H_{1}=\left\langle \left(1,2,3,...,n\right)\right\rangle \cong C_{n}$
and $H_{2}=stab_{S_{n}}(\left\{ n\right\} )\cong S_{n-1}$, hence
for $n=p$ prime we get that $\left[G:H_{1}\right]=(p-1)!$ and $\left[G:H_{2}\right]=p$
are coprime. The $H_{1}$-lattice $Fl(X)_{H_{1}}$ is just $Fl(C_{n},C_{n})$
which is a permutation lattice by \corref{flows_for_cyclic_groups},
since $H_{1}$ is cyclic. On the other hand, the $H_{2}$-lattice
$Fl(X)_{H_{2}}$ is again a flow in graph where there is a fixed vertex.
We already showed in \corref{fixed_vertex} that in such cases there
is a reduction to Noether's problem $\FF(\ZZ S_{n})^{S_{n}}/\FF$
which is known to be stably rational . A further investigation shows
that $Fl(X)_{H_{2}}$ is already in itself a permutation lattice.
Indeed, let $T$ be the subtree of $X$ consisting of the edges $\left\{ (i,n)\;\mid\;1\leq i\leq n-1\right\} $.
For each edge $e$ not in $T$ there is a unique way to complete it
to a simple flow in the graph using the edges of $T$, which we denote
by $C_{e}$. It is easily seen that $\left\{ C_{e}\;\mid\; e\notin T\right\} $
is a basis for $Fl(X)$ which is $H_{2}$-stable, so that $Fl(X)$
is an $H_{2}$-permutation lattice.\\

Finally, in \cite{beneish_induction_1998} Beneish proved the following.
For $p$ prime, let $H\leq S_{p}$ be the subgroup generated by the
cycle $(1,2,3,...,p)$ and let $N$ be its normalizer. Then $N$ contains
a cyclic group $C$ of order $p-1$ such that $N\cong H\rtimes C$
is isomorphic to the affine group of the finite field with $p$ elements.
Beneish showed that under these notations the fields $\CC(Fl(X))$
is stably isomorphic to $\CC(\ZZ S_{p}\otimes_{N}Fl(X))$ as $S_{n}$-fields
(and therefore also as $N$-fields). Considering $Fl(X)$ as a module
over $\ZZ N$, we get the set of flows on the graph with vertices
$\nicefrac{N}{C}$, which is exactly the type of graphs appearing
in \thmref{faithful_graphs} when we give a reduction for group with
cyclic $p$-Sylow subgroups. Thus, a better understanding of these
lattices over the affine group might lead to new results on the center
of the standard generic division algebra.

\bibliographystyle{my-style}
\bibliography{generic}

\begin{thebibliography}{10}

\bibitem{aljadeff_polynomial_????}
E.~Aljadeff and Y.~Karasik, \emph{Polynomial identities and {G}-graded
  {Azumaya} algebras}, Preprint .

\bibitem{aljadeff_crossed_2013}
E.~Aljadeff and Y.~Karasik, \emph{Crossed products and their central
  polynomials}, J. Pure Appl. Algebra \textbf{217}(9):(2013), 1634--1641.

\bibitem{aljadeff_polynomial_2008}
E.~Aljadeff and C.~Kassel, \emph{Polynomial identities and noncommutative
  versal torsors}, Adv. Math. \textbf{218}(5):(2008), 1453--1495.

\bibitem{amitsur_central_1972}
S.~A. Amitsur, \emph{On central division algebras}, Israel J. Math.
  \textbf{12}(4):(1972), 408--420.

\bibitem{amitsur_generic_1982}
S.~A. Amitsur, \emph{Generic splitting fields}, in F.~M. J.~v. Oystaeyen and
  A.~H. M.~J. Verschoren, editors, Brauer {Groups} in {Ring} {Theory} and
  {Algebraic} {Geometry}, number 917 in Lecture {Notes} in {Mathematics}, pages
  1--24, Springer Berlin Heidelberg (1982).

\bibitem{balaba_graded_2005}
I.~N. Balaba, \emph{Graded {Prime} {PI}-{Algebras}}, J. Math. Sci.
  \textbf{128}(6):(2005), 3345--3349.

\bibitem{beneish_induction_1998}
E.~Beneish, \emph{Induction theorems on the stable rationality of the center of
  the ring of generic matrices}, Trans. Amer. Math. Soc.
  \textbf{350}(9):(1998), 3571--3585.

\bibitem{bessenrodt_stable_1991}
C.~Bessenrodt and L.~Le~Bruyn, \emph{Stable rationality of certain
  {PGL}\_n-quotients}, Invent. Math. \textbf{104}(1):(1991), 179--199.

\bibitem{bloch_torsion_1974}
S.~Bloch, \emph{Torsion algebraic cycles, {K}\_2, and {Brauer} groups of
  function fields}, Bull. Amer. Math. Soc. (N.S.) \textbf{80}(5):(1974),
  941--945.

\bibitem{colliot-thelene_principal_1987}
J.-L. Colliot-Thélène and J.-J. Sansuc, \emph{Principal homogeneous spaces
  under flasque tori: {Applications}}, J. Algebra \textbf{106}(1):(1987),
  148--205.

\bibitem{endo_classification_1975}
S.~Endo and T.~Miyata, \emph{On a classification of the function fields of
  algebraic tori}, Nagoya Math J. \textbf{56}:(1975), 85--104.

\bibitem{fein_brauer_1981}
B.~Fein and M.~Schacher, \emph{Brauer groups of rational function fields over
  global fields}, in M.~Kervaire and M.~Ojanguren, editors, Groupe de {Brauer},
  number 844 in Lecture {Notes} in {Mathematics}, pages 46--74, Springer Berlin
  Heidelberg (1981).

\bibitem{fischer_isomorphie_1915}
E.~Fischer, \emph{Die {Isomorphie} der {Invariantenkorper} der endlichen
  {Abelschen} {Gruppen} lineaerer {Transformatioen}}, Nachrichten von der
  Gesellschaft der Wissenschaften zu Gottingen, Mathematisch-Physikalische
  Klasse pages 77--80.

\bibitem{formanek_center_1979}
E.~Formanek, \emph{The center of the ring of 3x3 generic matrices}, Linear and
  Multilinear Algebra \textbf{7}(3):(1979), 203--212.

\bibitem{formanek_center_1980}
E.~Formanek, \emph{The center of the ring of 4x4 generic matrices}, J. Algebra
  \textbf{62}(2):(1980), 304--319.

\bibitem{formanek_polynomial_1992}
E.~Formanek, The {Polynomial} {Identities} and {Invariants} of n x n
  {Matrices}, American Mathematical Society (1992).

\bibitem{hall_theory_1999}
M.~Hall, The theory of groups, AMS Chelsea Pub., Providence, R.I. (1999).

\bibitem{kang_retract_2012}
M.-c. Kang, \emph{Retract rational fields}, J. Algebra \textbf{349}(1):(2012),
  22--37.

\bibitem{katsylo_stable_1990}
P.~I. Katsylo, \emph{Stable rationality of the field of invariants of linear
  representations of the groups {PSL}6 and {PSL}12}, Mathematical notes of the
  Academy of Sciences of the USSR \textbf{48}(2):(1990), 751--753.

\bibitem{le_bruyn_permutation_????}
L.~Le~Bruyn, \emph{Permutation modules and rationality problems 1}.

\bibitem{le_bruyn_centers_1991}
L.~Le~Bruyn, \emph{Centers of generic division algebras, the rationality
  problem 1965-1990}, Israel J. Math. \textbf{76}(1):(1991), 97--111.

\bibitem{lenstra_rational_1974}
H.~W. Lenstra, \emph{Rational functions invariant under a finite abelian
  group}, Invent. Math. \textbf{25}(3):(1974), 299--325.

\bibitem{lorenz_multiplicative_2005}
M.~Lorenz, Multiplicative {Invariant} {Theory}, Springer, 2005 edition (2005).

\bibitem{merkurjev_cohomology_1982}
A.~S. Merkurjev and A.~A. Suslin, \emph{Cohomology of {Severi}-{Brauer}
  varieties and the norm residue homomorphism}, Izvestiya Rossiiskoi Akademii
  Nauk. Seriya Matematicheskaya \textbf{46}(5):(1982), 1011--1046.

\bibitem{meyer_separable_1971}
F.~D. Meyer and E.~Ingraham, Separable {Algebras} over {Commutative} {Rings},
  Springer, 1971 edition (1971).

\bibitem{orzech_brauer_1975}
M.~Orzech and C.~Small, The {Brauer} {Group} of {Commutative} {Rings}, M.
  Dekker, New York (1975).

\bibitem{procesi_non-commutative_1967}
C.~Procesi, \emph{Non-commutative affine rings}, Atti Accad. Naz. Lincei Mem.
  Cl. Sci. Fis. Mat. Natur. Sez. I (8) \textbf{8}:(1967), 237--255.

\bibitem{rosset_group_1978}
S.~Rosset, \emph{Group extensions and division algebras}, J. Algebra
  \textbf{53}(2):(1978), 297--303.

\bibitem{saltman_generic_1982}
D.~J. Saltman, \emph{Generic {Galois} extensions and problems in field theory},
  Adv. Math. \textbf{43}(3):(1982), 250--283.

\bibitem{saltman_retract_1984}
D.~J. Saltman, \emph{Retract rational fields and cyclic {Galois} extensions},
  Israel J. Math. \textbf{47}(2):(1984), 165--215.

\bibitem{saltman_note_1992}
D.~J. Saltman, \emph{A note on generic division algebras.}, Contemp. Math.
  \textbf{130}:(1992), 385--394.

\bibitem{saltman_lectures_1999}
D.~J. Saltman, Lectures on {Division} {Algebras}, American Mathematical Society
  (1999).

\bibitem{schofield_matrix_1992}
A.~Schofield, \emph{Matrix invariants of composite size}, J. Algebra
  \textbf{147}(2):(1992), 345--349.

\bibitem{snider_is_1979}
R.~Snider, \emph{Is the brauer group generated by cyclic algebras?}, in
  D.~Handelman and J.~Lawrence, editors, Ring theory, volume 734 of
  \emph{Lecture {Notes} in {Mathematics}}, pages 279--301, Springer Berlin /
  Heidelberg (1979).

\bibitem{swan_invariant_1969}
R.~G. Swan, \emph{Invariant rational functions and a problem of {Steenrod}},
  Invent. Math. \textbf{7}(2):(1969), 148--158.

\bibitem{sylvester_involution_1883}
J.~J. Sylvester, \emph{On the {Involution} of {Two} {Matrices} of the {Second}
  {Order}}, The Collected Mathematical Paper \textbf{4}.

\end{thebibliography}

\end{document}